\documentclass{article}

\usepackage{babel}
\usepackage[intlimits]{amsmath}
\usepackage{amsthm}
\usepackage{amssymb}
\usepackage{tikz-cd}
\usepackage{lmodern}
\usepackage{bm} 
\usepackage{enumerate} 
\usepackage{tasks}
\usepackage{subcaption}

\usepackage[margin=1in]{geometry} 
\usepackage[hidelinks,pdfpagemode=UseNone,pdfstartview=FitH]{hyperref}
\usepackage{amsrefs}

\theoremstyle{plain}
\newtheorem{theorem}{Theorem}[section]
\newtheorem{proposition}[theorem]{Proposition}
\newtheorem*{theorem*}{Theorem}
\newtheorem{corollary}[theorem]{Corollary}
\newtheorem{lemma}[theorem]{Lemma}

\theoremstyle{definition}
\newtheorem{definition}[theorem]{Definition}
\newtheorem{example}[theorem]{Example}
\newtheorem{remark}[theorem]{Remark}

\allowdisplaybreaks[1] 

\newcommand*{\Nset}{\mathbb{N}} 
\newcommand*{\Zset}{\mathbb{Z}} 
  
\newcommand*{\Rset}{\mathbb{R}}  
  
\newcommand\twoheaduparrow{\mathrel{\rotatebox[origin=c]{90}{$\twoheadrightarrow$}}}
\newcommand\twoheaddownarrow{\mathrel{\rotatebox[origin=c]{270}{$\twoheadrightarrow$}}}

\newcommand{\Mod}{\mathbf{Mod}}
\newcommand{\posetC}{\mathbf{P}}

\newcommand{\Sh}{\mathbf{Sh}}

\newcommand*{\grHom}{\underline{\Hom}} 
\newcommand*{\grExt}{\underline{\Ext}} 
\newcommand*{\grTor}{\underline{\Tor}} 
\newcommand*{\Eph}{\mathbf{Eph}} 
 
\newcommand*{\Trans}{\mathbf{Trans}} 
\newcommand*{\etal}{\emph{et al.\ }}

\DeclareMathOperator{\id}{id}
\DeclareMathOperator{\Hom}{Hom}

\DeclareMathOperator{\Fun}{Fun}
\DeclareMathOperator{\Open}{Open}
\DeclareMathOperator{\res}{res}
\DeclareMathOperator{\Ker}{Ker}
\DeclareMathOperator{\ImF}{Im}
\DeclareMathOperator{\Ext}{Ext}
\DeclareMathOperator{\Tor}{Tor}

\DeclareMathOperator{\op}{op}

\DeclareMathOperator{\inc}{inc}
\DeclareMathOperator{\loc}{loc}
\DeclareMathOperator{\sect}{sec}
\DeclareMathOperator{\supp}{supp}
\DeclareMathOperator{\Proj}{Proj}
\DeclareMathOperator{\Int}{Int}
\DeclareMathOperator{\cl}{cl}
\DeclareMathOperator{\Rad}{rad}
\DeclareMathOperator{\soc}{soc}
\DeclareMathOperator{\Ann}{Ann}
\DeclareMathOperator{\topf}{top}
\DeclareMathOperator{\coloc}{coloc}
\DeclareMathOperator{\cosec}{cosec}

\begin{document}

\title{Ephemeral Modules and Scott Sheaves on a Continuous Poset}
\author{Manu Harsu\footnote{Manu Harsu was supported by the Vilho, Yrjö and Kalle
Väisälä Foundation of the Finnish Academy of Science and Letters.
} \quad Eero Hyry \\
manu.harsu@tuni.fi \quad eero.hyry@tuni.fi}

\maketitle


\begin{abstract}
\noindent
By utilizing domain theory, we generalize the notion of an ephemeral module to the so-called continuous posets. 
We investigate the quotient category of persistence modules by the Serre subcategory of ephemeral modules and show that it is equivalent to the category of sheaves on the Scott topology. 
Furthermore, we study the metric properties of persistence modules via this equivalence. 
\end{abstract}


\section{Introduction}

This article is motivated by recent progress in topological data analysis.
Topological data analysis studies the shape of data using methods of algebraic topology. 
One of the main tools of topological data analysis is persistent homology. 
In persistent homology, one starts with a noisy data set, transforms it into a filtration of a topological space indexed by a poset, 
and then applies homology with coefficients in a field 
to obtain a diagram of vector spaces, which is often called a \emph{persistence module}. 
The homological features of the data that persist along the filtration are considered essential. 
Features that have zero ``lifespans'' in this sense are called \emph{ephemeral}. 
They are usually disregarded in topological data analysis. 
The notion of interleaving gives a way of measuring distances between persistence modules 
\cites{Bauer_2020, Bubenik_2014, desilva2018theory}. 
A persistence module over $\Rset$ can be regarded \emph{ephemeral} if its interleaving distance from the zero module vanishes. 
All this led Chazal \etal \cites{Obs, chazal2013structure} to investigate the quotient category of persistence modules by the Serre subcategory of ephemeral modules, which they called the category of \emph{observable modules}. 
Berkouk and Petit \cite{Berkouk_2021} defined ephemeral modules for $\Rset^n$. 
The main purpose of this article is to extend these investigations to the so-called continuous posets. 
Continuous posets include, for example, all discrete posets, as well as $\Rset^n$ and its down-sets. 
Also, products of continuous posets are continuous. 
Continuous posets are ubiquitous in the field of domain theory. 
Domain theory is a branch of order theory, which up to now, has mostly been used in logics and computer science. 
Domain theory provides abstract notions of approximation and convergence. 
We believe that this gives more insight even in the study of persistence over $\Rset^n$. 

Curry \cite{Curry1} introduced sheaf theoretic methods to topological data analysis. 
These methods have led to many fruitful advances (see, for example, \cites{Berkouk_2021, kashiwara2018persistent, Miller3}).
Curry equipped the poset with the Alexandrov topology.
We will consider sheaves on the Scott topology, which we call \emph{Scott sheaves} for short. 
On $\Rset^n$, Scott sheaves coincide with $\gamma$-sheaves introduced by Kashiwara and Schapira in \cite{SheavesOnManifolds}. 
In turn, on a discrete poset, the Scott topology is just the Alexandrov topology, making discrete posets uninteresting for our investigations. 
The interplay of persistence modules and Scott sheaves is a central subject in this article. 
In particular, we show that the category of Scott sheaves is equivalent 
to the quotient category of \emph{the category of persistence modules by the subcategory of ephemeral modules}.  
On the algebraic side, Scott sheaves turn out to correspond to semi-continuous modules. 
Semi-continuous modules emerge, 
for example, in symplectic topology.
Some authors (e.g.\ \cite{polterovich2015autonomous}) go so far as to restrict their study of persistence modules to semi-continuous modules. 

We now explain our results in more detail. 
Let $P$ be a poset. 
In domain theory, a poset is called continuous if every element is the directed supremum of the elements approximating it. 
The \emph{approximation} or \emph{way-below} relation $\ll$ is a binary relation on $P$, stronger than the usual $\le$. 
One says that $x$ approximates $y$ if, for all directed sets $D$ with $y \le \sup D$, there exists $d \in D$ such that $x \le d$.
We sometimes say that a directed set $D \subseteq P$ \emph{converges} to $p \in P$ if $\sup D = p$. 
An element $x \in P$ is called \emph{compact} or \emph{isolated from below} if $x \ll x$. 
Recall that a subset is open in the Alexandrov topology if and only if it is an up-set. 
In the Scott topology, a subset $U \subseteq P$ is declared to be open if it is an up-set and $D \cap U \not= \emptyset$ for all directed sets $D$ with $\sup D \in U$. 
In particular, the Scott topology is finer than the Alexandrov topology, 
implying that the identity map $j \colon P^a \to P^{\sigma}$ is continuous. 
Here $P^a$ and $P^{\sigma}$ denote the poset $P$ endowed with the Alexandrov topology and the Scott topology, 
respectively. 

Let $k$ be a commutative ring with unity. 
A persistence module is a functor from $P$ interpreted as a category to the category of $k$-modules $\Mod$. 
The category of persistence modules over $P$ is denoted by $\Fun(\posetC, \Mod)$. 
If $P$ is a partially ordered abelian group, we often consider a persistence module as a $P$-graded $k[U_0]$-module \cite[Theorem 2.21]{Bubenik_2021}. 
Here $U_0 = \{ x \in P \mid x \ge 0 \}$ is a basic up-set. 
In addition to persistence modules, 
we also investigate the categories of sheaves $\Sh(P^a)$ and $\Sh(P^{\sigma})$. 
It is well-known that the category $\Sh(P^a)$ is equivalent 
to the category $\Fun(\posetC, \Mod)$ \cite[Theorem 4.2.10]{Curry1}. 
We now observe in Proposition \ref{FullyFaithfulAdjointCombo} that if $P$ is a continuous poset, then the direct image functor $j_* \colon \Sh(P^a) \to \Sh(P^{\sigma})$ has, besides the left adjoint $j^*$, also a right adjoint $j^!$. 
The situation can be summarized as follows:  
\[
\begin{tikzcd}
  \Fun(\posetC, \Mod) \arrow{r}{\simeq} & 
  \Sh(P^a) \arrow{l} \arrow[r,"{j_*}" description]  & 
  \Sh(P^{\sigma}). \arrow[l,"{j^*}" description, bend right] \arrow[l,"{j^!}" description, bend left] 
\end{tikzcd}
\]

For a persistence module $M$ over $P$ and $p \in P$, we set 
\[
  \underline{M}_p = \varprojlim_{x \gg p} M_x \quad \text{and} \quad \overline{M}_p = \varinjlim_{x \ll p} M_x. 
\]
This gives us two functors $\underline{(-)}, \, \overline{(-)} \colon \Fun(\posetC, \Mod) \to \Fun(\posetC, \Mod)$. 
In the case of a linearly ordered poset, these functors have been studied, for example, in \cites{Obs, Schmahl_2022, Scoccola2020LocallyPC}. 
When the poset is a real polyhedral group, the persistence modules $\overline{M}$ and $\underline{M}$ are also related to the upper and lower boundaries introduced by Miller in \cite[Definitions 3.11 and 11.1]{Miller1} (see Remark \ref{Rmk:UpperLowerBoundary}). 
Given a convex subset $I \subseteq P$, the \emph{indicator module} $k[I]$ assigns $k$ to every $p \in I$ with identity maps on $k$. 
Interestingly, it turns out in Proposition \ref{prop:IntervalIntClosure} that in the case of the up-set and down-set indicator modules, the functors $\underline{(-)}$ and $\overline{(-)}$ correspond to taking the topological interior and closure (cf.\  \cite[Lemma 3.20]{Miller1}). 
The persistence module $M$ is \emph{upper semi-continuous} (resp.\ \emph{lower semi-continuous}), 
if the canonical morphism $M \to \underline{M}$ (resp.\ $\overline{M} \to M$) is an isomorphism. 
In Theorem \ref{EqvUpperSemiCont}, we show that the full subcategories of upper semi-continuous modules and lower semi-continuous modules are both equivalent to the category of Scott sheaves. 

An immediate example of lower semi-continuity is provided by finitely Scott encoded modules (see Proposition \ref{FinScottEncodedAreLSemiCont}). 
We call here a persistence module $M$ over $P$ \emph{finitely Scott encoded} 
if there is a Scott-continuous map $f \colon P \to F$ to a finite poset $F$ 
and a pointwise finite-dimensional persistence module $N$ over $F$ (i.e.\ $N_x$ is finite-dimensional for all $x \in F$) 
such that $M = f^* N$. 
Note that a map $f \colon X \to \Rset$, where $X$ is a topological space, is \emph{lower semi-continuous} in the classical sense if and only if $f$ is continuous with respect to the Scott topology of $\Rset$. 

We say that a persistence module $M$ is \emph{ephemeral} if $M(p \le q) = 0$ for all $p \ll q$ (Definition \ref{def:Eph}). 
We give in Theorem \ref{meagerSupp} a new characterization for ephemeral modules in terms of meager sets. 
We call a subset $S \subseteq P$ is \emph{meager} if $x \not\in S$ or $y \not\in S$ for all $x \ll y$. 
For example, a \emph{staircase} ({\cite[p.\ 51]{Miller4}}) i.e.\ the topological boundary of the set 
\[
  \{ \bm{x} \in \Rset^n \mid \bm{x} \ge \bm{v} \text{ for some } \bm{v} \in V \}, 
\]
where $V \subseteq \Nset^n$ is a set of pairwise incomparable elements, is meager. 
A persistence module $M$ over $P$ is now ephemeral 
if and only if the support $\supp(m)$ is meager for all $p \in P$ and $m \in M_p$. 

In Theorem \ref{EphWayBelow}, we make the crucial observation that over a continuous poset $P$, 
a persistence module $M$ is ephemeral if and only if $j_* M = 0$. 
In fact, this was used as the definition of an ephemeral module in \cite{Berkouk_2021}. 
Right away, it follows that the full subcategory of ephemeral modules $\Eph$ is a Serre subcategory of the category of persistence modules. 
This result, with the general framework of Serre subcategories and the corresponding quotient categories, 
permits us to consider the quotient category of \emph{the category of persistence modules by the subcategory of ephemeral modules}. 
In Theorem \ref{Eqv1}, 
we prove that this quotient is equivalent to the category of Scott sheaves. 
As a consequence, it is also equivalent with the categories of lower semi-continuous and upper semi-continuous modules. 
We observe in Proposition \ref{Prop:IndecomInjScottSobr2} that in the case of $\Rset^n$, the indecomposable injectives of the quotient category are in a one-to-one correspondence with the upward-directed down-sets that are open with respect to the standard topology. 

Miller realized in \cite{Miller2} that the socles and tops of persistence modules can be used to functorially describe deaths and births of the homological properties of data. 
We propose new variants for these notions by defining for a persistence module $M$ over $P$ the \emph{Scott-socle} and the \emph{Scott-top}, denoted by $\soc^{\sigma}(M)$ and $\topf^{\sigma}(M)$, respectively (see Definition \ref{Def:RadSocTop}). 
At $p \in P$, the classical socle detects the elements of $M_p$ that vanish in all internal morphisms $M(p \le x)$ with $x > p$. 
In turn, the Scott-socle consists of those element of $M_p$ that vanish in internal morphisms $M(p \le x)$ with $x \gg p$. 
For example, consider the indicator module of the basic-down set $D_{\bm{p}} = \{ \bm{x} \in \Rset^n \mid \bm{x} \le \bm{p} \}$. 
The Scott-socle is now the indicator module $k[\partial D_{\bm{p}}]$ associated to the topological boundary of $D_{\bm{p}}$, whereas the usual socle is the skyscraper module $k_{\bm{p}}$ (see Proposition \ref{Prop:RadSocTopInterval} i)). 

We are also interested in their derived functors. 
In Proposition \ref{Prop:SemiContHomCrit}, we show that $M$ is upper semi-continuous if and only if $\soc^{\sigma}(M) = 0 = R^1 \soc^{\sigma}(M)$. 
Similarly, $M$ is lower semi-continuous if and only if $\topf^{\sigma}(M) = 0 = L_1\topf^{\sigma}(M)$. 
In Theorem \ref{Thm:SocTopConnection}, we can prove the formulas 
\[
  R^1 \soc^{\sigma}(\overline{M}) = \topf^{\sigma}(\underline{M})
  \quad \text{and} \quad 
  L_1 \topf^{\sigma}(\underline{M}) = \soc^{\sigma}(\overline{M}). 
\]
Curiously, if $P$ is also a partially ordered abelian group, then 
there is an exact sequence 
\[
  0 \to \soc^{\sigma}(M) \to M \to \bigoplus_{p \in P} \Gamma(\Int U_0, j_*M(p)) \to R^1\soc^{\sigma}(M) \to 0
\]
and isomorphisms 
\[
  R^{n + 1} \soc^{\sigma}(M) = \bigoplus_{p \in P} H^n(\Int U_0, j_*M(p)),  
\] 
for all $n \ge 1$, mimicking the celebrated Serre-Grothendieck correspondence between sheaf cohomology and local cohomology (see Proposition \ref{Prop:HereditarySoc} and Remark \ref{Rmk:LocalCoh}). 
In fact, by using a result of Jensen, we observe in Proposition \ref{Prop:SocVanish} that the functor $R^n \soc^{\sigma}$ vanishes for $n > 1$ under mild assumptions. 

Finally, in Definition \ref{ScottInterlea}, we define the interleaving distance for Scott sheaves.
To fully justify the transition to the quotient category $\Fun(\posetC, \Mod) / \Eph$, the quotient functor should then be an isometry 
i.e.\ $d_a(M, N) = d_{\sigma}(j_* M, j_* N)$ for all persistence modules $M$ and $N$ over $P$, 
where $d_a$ and $d_{\sigma}$ denote the interleaving distance of persistence modules 
and of Scott sheaves, respectively.
This is proved in Theorem \ref{Isom1}. 
The other functors $j^*, j^!, \underline{(-)}$ and $\overline{(-)}$ are shown to be isometries, too. 
We also investigate how ephemeral modules fit into this setting. 
In Corollary \ref{Cor:Isom}, we prove that a persistence module $M$ is ephemeral if and only if $d_a(M, 0) = 0$.

\section{Preliminaries on Domain Theory}\label{sec:Preliminary}

For a general reference on domain theory, see \cites{AbramskyJung, gierz2003, goubault-larrecq_2013}. 
Every poset $P$ can be interpreted as a category. 
The objects are the elements of $P$.  
For $p,q \in P$, the set of morphisms $\Hom(p,q)$ is just a singleton, if $p \le q$, and empty otherwise. 
We denote this category by $\posetC$. 

A subset $U \subseteq P$ is called an \emph{up-set} if $x \in U$ and $x \le y$ in $P$ imply $y \in U$. 
For every $p \in P$, there is the \emph{basic up-set} $U_p = \{ x \in P \mid x \ge p \}$. 
Dually, a subset $D \subseteq P$ is a \emph{down-set} if $x \in D$ and $y \le x$ in $P$ imply $y \in D$. 
For every $p \in P$, we have the \emph{basic down-set} $D_p = \{ x \in P \mid x \le p \}$. 
We say that a subset $D \subseteq P$ is \emph{(upward)-directed}, 
if it is non-empty, and for all $x,y \in D$ there exists $z \in D$ such that $x,y \le z$. 
Dually, a subset $D$ is \emph{downward-directed} 
if it is non-empty and for all $x,y \in D$ there exists $z \in D$ such that $z \le x,y$. 
A subset $C \subseteq P$ is \emph{convex} if $x,z \in C$ and $x \le y \le z$ in $P$ implies $y \in C$. 
A subset $C \subseteq P$ is \emph{connected} if, for all $x,y \in C$, there are $z_1, \dots, z_n \in C$ 
such that $z_1 = x$, $z_n = y$ and for all $i = 1, \dots, n - 1$ either $z_i \le z_{i + 1}$ or $z_{i} \ge z_{i + 1}$. 
The connected convex subsets are called \emph{intervals}. 

In the heart of domain theory, there is the so-called \emph{way-below relation}. 
For $x,y \in P$, we say that $x$ is \emph{way below} $y$ (or $y$ is \emph{way above} $x$)
if for all directed sets $D$ with $y \le \sup D$ (whenever the supremum exists) there exists $d \in D$ such that $x \le d$. 
Sometimes we say that $x$ \emph{approximates} $y$. 
We write $x \ll y$. 
An element $x \in P$ is called \emph{compact} (or \emph{isolated from below}) if $x \ll x$. 
The way-below relation has the following basic properties: 
\begin{enumerate}[ i)]
  \item if $x \ll y$, then $x \le y$; 
  \item if $x \le y$ and $y \ll z$, then $x \ll z$; 
  \item if $x \ll y$ and $y \le z$, then $x \ll z$; 
\end{enumerate}
for all $x,y,z \in P$ 
(see \cite[Proposition 5.1.4]{goubault-larrecq_2013}). 

The way-below relation is fully understood in many posets. 
The above basic properties imply that if every element is compact, then the way-below relation is the normal partial order relation. 
This includes all discrete posets. 
On the other hand, the way-below relation of $\Rset^n$ is given by the equivalence 
\[
    \bm{x} \ll \bm{y} \iff x_i < y_i \text{ for all } i \quad \quad (\bm{x}, \bm{y} \in \Rset^n). 
\]
However, the way-below relation in a subposet can be much more complicated. 
One such example is given in \cite[Proposition 5.1]{Keimel_2009}. 
For example, in the case of the positive orthant $\Rset^n_{\ge 0}$, there is an equivalence 
\[
  \bm{x} \ll \bm{y} \iff x_i < y_i \text{ or } x_i = 0 \text{ for all } i. 
\]
Thus, $(1,0) \ll (2,0)$ in $\Rset^2_{\ge 0}$ 
even though $(1,0)$ is \emph{not} way below $(2,0)$ in $\Rset^2$. 
When we restrict to the subposet, 
a directed set cannot converge to a boundary point (i.e.\ point that have zero in at least one coordinate) 
from outside of the subposet, which explains the difference. 

More generally, we can consider partial orders of $\Rset^n$ induced by cones. 
For a general reference about cones, see, for example, \cite{rockafellarConvex}.
Given a closed convex proper cone $\gamma \subseteq \Rset^n$, we define a partial order on $\Rset^n$ by setting 
\[
  \bm{x} \le_{\gamma} \bm{y} \iff \bm{y} + \gamma \subseteq \bm{x} + \gamma \iff \bm{y} - \bm{x} \in \gamma
\]
for all $\bm{x}, \bm{y} \in \Rset^n$. 
Domain theoretic properties of the poset $(\Rset^n, \le_{\gamma})$ is well-understood. 
For example, if $\gamma$ has a non-empty interior, then the corresponding way-below relation satisfies 
\[
  \bm{x} \ll_{\gamma} \bm{y} \iff \bm{y} - \bm{x} \in \Int \gamma
\]
for all $\bm{x}, \bm{y} \in \Rset^n$, where $\Int \gamma$ denotes the interior of $\gamma$ with respect to the standard topology (see \cite[Section 4]{Keimel_2009}). 

Given a subset $S \subseteq P$, we set 
\begin{align*}
  {\uparrow} S &= \{ x \in P \mid x \ge p \text{ for some } p \in S \} & {\downarrow} S &= \{ x \in P \mid x \le p \text{ for some } p \in S \} \\
  {\twoheaduparrow} S &= \{ x \in P \mid x \gg p \text{ for some } p \in S \} & {\twoheaddownarrow} S &= \{ x \in P \mid x \ll p \text{ for some } p \in S \}. 
\end{align*}
In particular, we use notation 
\begin{align*}
  {\uparrow} p &= \{ x \in P \mid x \ge p \} & {\downarrow} p &= \{ x \in P \mid x \le p \} \\
  {\twoheaduparrow} p &= \{ x \in P \mid x \gg p \} & {\twoheaddownarrow} p &= \{ x \in P \mid x \ll p \}. 
\end{align*}
for all $p \in P$. 
The basic properties of the way-below relation imply that ${\twoheaduparrow} p$ is an up-set 
and ${\twoheaduparrow} p \subseteq  {\uparrow} p$. 
Dually, ${\twoheaddownarrow} p$ is a down-set, and ${\twoheaddownarrow} p \subseteq {\downarrow} p$. 
However, in the sequel, we will write $U_p$ and $D_p$ for the basic up- and down-sets. 
With these notations, we now define a poset $P$ to be \emph{continuous}, 
if for all $p \in P$ the set ${\twoheaddownarrow} p$ is upward-directed with supremum $p$. 
Continuous posets have the so-called \emph{interpolation property}: 
\[
  \text{for all } x \ll z \text{ there exist } y \in P \text{ such that } x \ll y \ll z 
\]
(see \cite[Proposition 5.1.15]{goubault-larrecq_2013}). 
This property is foundational for us, as all our major results follow from it in one way or another. 

There are plenty of continuous posets. 
For example, all discrete posets and linearly ordered posets (see \cite[Example 2.6]{XuContPoset}) are continuous. 
The continuous posets are also closed under products (see \cite[Proposition 5.1.53]{goubault-larrecq_2013}). 
For a closed convex proper cone $\gamma \subseteq \Rset^n$ with a non-empty interior, 
the poset $(\Rset^n, \le_{\gamma})$ is continuous (see \cite[Section 4]{Keimel_2009}). 
Especially, $\Rset^n$ is a continuous poset. 
The main result of the article \cite{Poncet_2022} (see especially p.\ 110) provides even more examples of continuous posets. 
In particular, every down-set of a continuous meet-semilattice is continuous itself. 
This also implies that every down-set of $\Rset^n$ is continuous.

Let us now move to consider two important topologies of a poset $P$. 
In the \emph{Alexandrov topology}, 
the open sets are the up-sets, and the closed sets are the down-sets. 
The collection of the basic up-sets $\{ U_p \}_{p \in P}$ is its natural basis. 
The second important topology is the Scott topology. 
The Scott topology of $P$ is now defined by declaring a subset $U \subseteq P$ to be \emph{Scott-open}, if 
\begin{enumerate}[ i)]
  \item $U$ is an up-set; 
  \item for every directed subset $D$, the intersection $D \cap U$ is non-empty whenever $\sup D$ exists and is in $U$; 
\end{enumerate}
(see \cite[Proposition 4.2.18]{goubault-larrecq_2013}). 
We use the notations $P^a$ and $P^{\sigma}$ for $P$ endowed with the Alexandrov topology and the Scott topology, respectively.  
Since every Scott-open set is an up-set, the Alexandrov topology is finer than the Scott topology. 
In particular, the identity map $j \colon P^a \to P^{\sigma}$ is continuous. 

A map $f \colon P \to Q$ between posets is order-preserving 
if and only if the corresponding map $f^a \colon P^a \to Q^a$ is continuous \cite[Lemma 2.4]{Bubenik_2021}. 
In addition, a map $f \colon P \to Q$ between posets is called \emph{Scott-continuous} 
if it is order-preserving and for all directed subsets $D \subseteq P$ we 
have $f(\sup D) = \sup f(D)$ whenever the supremum $\sup D$ exists. 
In \cite[Proposition 4.3.5]{goubault-larrecq_2013}, it is proven that a map $f \colon P \to Q$ is Scott-continuous 
if and only if the corresponding map $f^{\sigma} \colon P^{\sigma} \to Q^{\sigma}$ is continuous.

\begin{remark}\label{Rmk:DpScottClosed}
Let $P$ be a poset. 
A subset $F \subseteq P$ is closed in the Scott topology 
if and only if it is down-set and closed under directed supremums. 
In other words, if $D$ is directed and $D \subseteq F$, then $\sup D \in F$ whenever the supremum exists. 
The basic down-set $D_p$ is Scott-closed for every $p \in P$ (\cite[Remark II-1.4]{gierz2003}). 
\end{remark}

Even though the Scott topology can be defined for any poset, it is convenient to restrict to continuous posets. 
The interpolation property implies the following facts:

\begin{lemma}[{\cite[Proposition II-1.10]{gierz2003}}]\label{Fact:IntU}
Let $P$ be a continuous poset. 
\begin{enumerate}[ i)]
  \item An up-set $U \subseteq P$ is Scott-open if and only if for all $p \in U$ there is $x \in U$ such that $x \ll p$; 
  \item The collection $\{ {\twoheaduparrow} p \}_{p \in P}$ is a basis of the Scott topology; 
  \item With respect to the Scott topology, we have $\Int({\uparrow}p) = {\twoheaduparrow} p$ for all $p \in P$; 
  \item With respect to the Scott topology, 
    we have  $\Int U = \bigcup_{x \in U} {\twoheaduparrow} x$ for all up-sets $U \subseteq P$. 
\end{enumerate}
\end{lemma}

It is easy to see that a set is open in the Scott topology of $\Rset^n$
if and only if it is open in the Alexandrov topology and in the standard topology. 
In addition, we have two observations regarding the Scott topology of $\Rset^n$:

\begin{remark}
Let $P = \Rset^n$. 
For the sake of clarity, we will in this remark 
denote the interiors with respect to the Scott topology and the standard topology by $\Int^{\sigma}$ and $\Int^s$, respectively. 
By Lemma \ref{Fact:IntU} iv), we have $\Int^{\sigma} U = {\twoheaduparrow} U$ for all up-sets $U \subseteq \Rset^n$. 
This set is open with respect to the standard topology. 
In particular, $\Int^{\sigma} U \subseteq \Int^s U$. 
The converse inclusion is also true, 
because the set ${\twoheaduparrow} B(\bm{x}, r)$ is open with respect to the Scott topology for all $\bm{x} \in \Rset^n$ and $r > 0$. 
Note that we also have $\Int^{s} D = {\twoheaddownarrow} D$ for all down-sets $D \subseteq \Rset^n$. 

Similar results exist for the closures. 
For example, $\cl^{\sigma} D = \cl^s{D}$ for all down-sets $D \subseteq \Rset^n$. 
Here $\cl^{\sigma}$ and $\cl^s$ are the closures with respect to the Scott topology and standard topology, respectively. 
\end{remark}

\begin{remark}\label{Rmrk:RnIntClosBound}
Let $P = \Rset^n$. 
Observe that the interpolation property gives the formulas 
\[
  {\twoheaduparrow} \overline{U} = {\twoheaduparrow} U 
  \quad \text{and} \quad 
  {\twoheaddownarrow} \overline{D} = {\twoheaddownarrow} D. 
\]
for all up-sets $U \subseteq \Rset^n$ and down-sets $D \subseteq \Rset^n$, 
where the closures are taken with respect to the standard topology. 
Thus, for the boundaries with respect to the standard topology, we have 
\[
  \partial U = \overline{U} \setminus {\twoheaduparrow} U 
  \quad \text{and} \quad 
  \partial D = \overline{D} \setminus {\twoheaddownarrow} D. 
\]
It is also worth noting that 
\[
  \partial \overline{U} 
  = \overline{U} \setminus {\twoheaduparrow} \overline{U}
  = \overline{U} \setminus {\twoheaduparrow} U
  = \partial U 
  \quad \text{and} \quad 
  \partial \overline{D} 
  = \overline{D} \setminus {\twoheaddownarrow} \overline{D} 
  = \overline{D} \setminus {\twoheaddownarrow} D 
  = \partial D. 
\]
\end{remark}

\section{Persistence Modules and Sheaves}\label{PerModAndSheaves}

\subsection{Persistence Modules}

Let $P$ be a poset, let $k$ be a commutative ring with unity, and let $\Mod$ denote the category of $k$-modules. 
The (covariant) functors $M \colon \posetC \to \Mod$ are called \emph{persistence modules}. 
Thus, for every $p \in P$, there is a $k$-module $M_p$, and for every relation $p \le q$, there is a $k$-homomorphism 
$M(p \le q) \colon M_p \to M_q$, which is called an \emph{internal morphism}. 
These morphisms satisfy 
\[
  M(p \le p) = \id \quad \text{and} \quad M(q \le r) \circ M(p \le q) = M(p \le r)
\]
for all $p \le q \le r$ in $P$. 
The persistence modules form a category, where the morphisms are just natural transformations. 
A natural transformation $f \colon M \to N$ between persistence modules 
is a family of $k$-homomorphisms $(f_p)_{p \in P}$ that makes the diagram 
\[
\begin{tikzcd}
  M_p \arrow{d}{f_p} \arrow{r}{M(p \le q)} & M_q \arrow{d}{f_q} \\
  N_p \arrow{r}{N(p \le q)} & N_q
\end{tikzcd}
\]
commute for all $p \le q$ in $P$. 
The category of persistence modules is denoted by $\Fun(\posetC, \Mod)$. 
It is an abelian category and, in fact, a Grothendieck category \cite[Proposition 2.23]{Bubenik_2021}. 
All the important objects (kernels, cokernels, images, products, direct sums) are defined pointwise. 
For example, the kernel of a morphism $f \colon M \to N$ is a persistence module $\Ker f$ 
with $(\Ker f)_p = \Ker f_p$ for all $p \in P$. 
Similarly, a morphism of persistence modules $f \colon M \to N$ is a monomorphism (resp. epimorphism) 
if and only if every $f_p \colon M_p \to N_p$ is monomorphism (resp. epimorphism) for all $p \in P$. 
The \emph{support} of a persistence module $M$ is the set 
$\supp M = \{ p \in P \mid M_p \not= 0 \}$. 

We now recall the definition of tensor product. 
For more details, see for example \cite[Chapter 15]{RichterCat}. 
Let $P$ be a poset, and let $N \colon \posetC^{\op} \to \Mod$ and 
$M \colon \posetC \to \Mod$ be persistence modules. 
Their \emph{tensor product} is the $k$-module $N \otimes_{\posetC} M$ that is defined as the quotient 
\[
  \bigoplus_{p \in P} N_p \otimes M_p \ \Big/ \ \Bigr\langle N(x \le y)(n) \otimes m - n \otimes M(x \le y)(m) 
    \ \Bigr| \ n \in N_y, m \in M_x \Bigr\rangle. 
\]
This operation defines a bifunctor 
\[
  - \otimes_{\posetC} - \colon \Fun(\posetC^{\op}, \Mod) \times \Fun(\posetC, \Mod) \to \Mod. 
\]
The functors $- \otimes_{\posetC} M$ and $N \otimes_{\posetC} - $ are right exact, and we denote by $\Tor_n(-, -)$ their left derived functors.

Now we give an easy and important example of persistence modules.

\begin{example}
Let $P$ be a poset, and let $I \subseteq P$ be a convex set. 
The \emph{indicator module} $k[I]$ is the persistence module defined by setting 
\[
  k[I]_p = 
    \begin{cases}
      k, &\text{if } p \in I; \\
      0, &\text{otherwise}. 
    \end{cases}
\]
The internal morphism $k[I](p \le q)$ is the identity if $p,q \in I$, and zero otherwise. 
\end{example}

Recall that a persistence module $M$ over a poset $P$ is \emph{finitely generated} 
if there is an exact sequence $F \to M \to 0$ 
for some persistence module $F$ over $P$ that is isomorphic to a finite direct sum of 
basic up-set indicator modules $k[U_x]$ ($x \in P$). 
Dually, $M$ is \emph{finitely co-generated} if there is an exact sequence $0 \to M \to E$ 
for some persistence module $E$ over $P$ that is isomorphic to a finite direct sum of 
basic down-set indicator modules $k[D_x]$ ($x \in P$).

A tuple $(P, \le, +, 0)$ is called a \emph{partially ordered abelian group}, if the pair $(P, \le)$ is a poset, the triple $(P, +, 0)$ is an abelian group, and $x \le y$ implies $x + z \le y + z$ for all $x,y,z \in P$.
For a partially ordered abelian group, there exists an equivalence between the category of persistence modules and the category of $P$-graded $k[U_0]$-modules (see \cite[Theorem 2.21]{Bubenik_2021}).
By means of this equivalence, 
we often identify a persistence module $M$ with the $P$-graded $k[U_0]$-module 
$\bigoplus_{p \in P} M_{p}$. 
The scalar multiplication by a monomial $x^{s} \cdot \colon M_p \to M_{p + s}$ ($s \in U_0$) 
comes from the internal morphism $M(p \le p + s)$. 

Let $M$ and $N$ be $P$-graded $k[U_0]$-modules. 
For $s \in P$, we denote by $M(s)$ the $P$-graded $k[U_0]$-module such that $M(s)_p = M_{p + s}$ for all $p \in P$. 
The \emph{graded} $\Hom$ is the $P$-graded $k[U_0]$-module $\grHom_{k[U_0]}(N, M)$ that has components 
\[
  \grHom_{k[U_0]}(N, M)_s = \Hom(N, M(s)) = \Hom(N(-s), M) 
\]
for all $s \in P$. 
We denote by $\grExt_{k[U_0]}^n$ $(n \in \Nset)$ the right derived functors of $\grHom_{k[U_0]}$. 

Let $M$ and $N$ be $P$-graded $k[U_0]$-modules. 
The tensor product $N \otimes_{k[U_0]} M$ is naturally $P$-graded $k[U_0]$-module. 
The grading is given by 
\[
  (N \otimes_{k[U_0]} M)_{s} = \langle n \otimes m \mid n, m \text{ homogeneous and } \deg n + \deg m = s \rangle 
\]
for all $s \in P$. 
The left derived functors are $\grTor_n^{k[U_0]}$ $(n \in \Nset)$. 

If $N^{\op}$ denotes the persistence module over $P^{\op}$ defined by setting $N^{\op}_p = N_{-p}$ for all $p \in P$, 
then 
\[
  (N \otimes_{k[U_0]} M)_s 
  = N^{\op} \otimes_{\posetC} M(s) 
  = N^{\op}(-s) \otimes_{\posetC} M 
\]
for all $s \in P$.

Let $X$ be a topological space. 
We use the notation $\Sh(X)$ for the category of sheaves on $X$ with values in $\Mod$. 
It is an abelian category. 
For this, and other basic facts concerning sheaves, see \cite{Hart}.
Because we will use the notion of a \emph{sheaf on a basis} several times in this article, we give a short summary. 
Let $\mathcal{B}$ be a basis of $X$. 
It is a poset ordered by the inclusion. 
A contravariant functor $\mathcal{F} \colon \mathcal{B} \to \Mod$ is called a $\mathcal{B}$-\emph{sheaf} 
if it satisfies the gluing condition. 
Note that in the gluing condition, 
one considers all sets $B, B_1, B_2 \in \mathcal{B}$ satisfying $B \subseteq B_1 \cap B_2$, 
since the intersection $B_1 \cap B_2$ might not belong to the basis $\mathcal{B}$. 
Sheaves and $\mathcal{B}$-sheaves are essentially the same thing. 
Firstly, every sheaf defines a $\mathcal{B}$-sheaf by forgetting all the unnecessary section modules. 
Conversely, a $\mathcal{B}$-sheaf defines a unique sheaf, 
whose section modules are defined by the inverse limits 
\[
  \mathcal{F}(U) = \varprojlim_{\mathcal{B} \ni B \subseteq U} \mathcal{F}(B), 
\]
for an open subset $U \subseteq X$.

Let $P$ be a poset. 
To differentiate the sheaves on the Alexandrov topology from the sheaves on the Scott topology,
we call the former \emph{Alexandrov sheaves} and the latter \emph{Scott sheaves}. 
Let us start with the Alexandrov sheaves. 
Every point $p \in P$ has the smallest open neighborhood, namely $U_p$. 
This makes the stalks very simple, $\mathcal{F}_p = \mathcal{F}(U_p)$
for all (pre)sheaves $\mathcal{F}$ on $P^a$. 
The collection $\{ U_p \}_{p \in P}$ being a basis of the Alexandrov topology, 
the sections $\{ \mathcal{F}(U_p) \}_{p \in P}$ are enough to determine the sheaf. 
Thus, a sheaf on Alexandrov topology and a $\{ U_p \}$-sheaf are fundamentally the same thing. 
This fact can be seen as an equivalence of categories: 
The category of persistence modules $\Fun(\posetC, \Mod)$ 
and the category of sheaves on the Alexandrov topology $\Sh(P^a)$ are equivalent. 
For a proof of this equivalence, see \cite[Theorem 4.2.10]{Curry1}. 
We omit the details, but present here the functors that define the equivalence. 
Firstly, a sheaf $\mathcal{F}$ is mapped to the persistence module, which at $p$ is the stalk $\mathcal{F}_p$. 
The internal morphisms are just restrictions. 
Conversely, a persistence module $M$ corresponds to the sheaf $M^{\dagger}$, 
whose sections over an up-set $U \subseteq P$ are limits 
\[
  M^{\dagger}(U) 
  = \varprojlim_{x \in U} M_x 
  = \Biggl\{ (m_x)_{x \in U} \in \prod_{x \in U} M_x \ \Big\vert \ M(x \le y)(m_x) = m_y \Biggr\}. 
\]
Restrictions are obvious. 
Later, we will use this equivalence without a specific mention. 
This means that we will write $M(U)$ instead of $M^{\dagger}(U)$. 
Most of the time, we prefer to speak about persistence modules, because it prevents confusion with the two sheaf categories. 

Let $P$ be a poset. 
The category of persistence modules over $P$ is also equivalent 
to the category of cosheaves on the Alexandrov topology of $P^{\op}$ (see, \cite[Theorem 4.2.10]{Curry1}). 
Note that the down-sets of $P$ are the open sets of $P^{\op}$. 
In this equivalence, a persistence module $M$ over $P$ is mapped to a cosheaf $\hat{M}$, whose cosection modules are colimits 
\[
  \hat{M}(D) = \varinjlim_{x \in D} M_x
\]
for all down-sets $D$ of $P$.

\begin{lemma}\label{Lemma:SheafCosheaf}
Let $P$ be a poset, and let $M$ be a persistence module over $P$. 
Then 
\begin{enumerate}
  \item[ i)] $M^{\dagger}(U) = \Hom(k[U], M)$ for all up-sets $U \subseteq P$; 
  \item[ ii)] $\hat{M}(D) = k[D] \otimes_{\posetC} M$ for all down-sets $D \subseteq P$. 
\end{enumerate}
Here $k[D]$ denotes the indicator module over $P^{\op}$ with the support $D$. 
\end{lemma}

\begin{proof}
Item i) follows directly from the definitions.
It has also been stated in \cite[Remark 7.2]{CohBBR}.
For ii), the tensor product $k[D] \otimes_{\posetC} M$ is the quotient 
\[
  \bigoplus_{p \in D} M_p \ \Big/ \ \Bigr\langle m_y - M(x \le y)(m_x) 
    \ \Bigr| \ m_x \in M_x, m_y \in M_y \Bigr\rangle. 
\]
But this is exactly the concrete description of the colimit $\varinjlim_{x \in D} M_x$. 
\end{proof}

\subsection{The Adjoint Triple $(j^*, j_*, j^!)$}\label{sec:functors}

Let $P$ be a poset. 
We next study, how we can travel between the two sheaf categories. 
Since the identity $j \colon P^a \to P^{\sigma}$ is continuous, 
there are the classical sheaf theoretic direct and inverse image functors
\[
  j_* \colon \Sh(P^{a}) \to \Sh(P^{\sigma}) \quad \text{and} \quad j^* \colon \Sh(P^{\sigma}) \to \Sh(P^{a}). 
\]
Note that these functors satisfy the formulas 
\[
  (j_*M)(\Int U_p) = M(\Int U_p) \quad \text{and} \quad 
  (j^*\mathcal{F})_p = \mathcal{F}_p
\]
for all $p \in P$. 
Due to the equivalence of persistence modules and and Alexandrov sheaves, 
we can also consider $j_*$ and $j^*$ as functors $\Fun(\posetC, \Mod) \to \Sh(P^{\sigma})$ and $\Sh(P^{\sigma}) \to \Fun(\posetC, \Mod)$, respectively. 
We now define a third functor 
\[
  j^! \colon \Sh(P^{\sigma}) \to \Fun(\posetC, \Mod) 
\]
by setting 
\[
  (j^! \mathcal{F})_p = \mathcal{F}(\Int U_p) 
\]
for all Scott sheaves $\mathcal{F}$ on $P^{\sigma}$ and $p \in P$. 
The internal morphisms are just restrictions. 
This functor is left exact, since it is defined by sections.

We can describe the sections of $j^! \mathcal{F}$ as follows:

\begin{proposition}\label{Prop:j!Sections}
Let $P$ be a continuous poset. 
For every Scott sheaf $\mathcal{F}$ on $P^{\sigma}$ and every up-set $U \subseteq P$, 
we have $(j^!\mathcal{F})(U) = \mathcal{F}(\Int U)$. 
\end{proposition}

\begin{proof}
By Lemma \ref{Fact:IntU} the collection $\{ \Int U_p \}_{p \in P}$ is a basis of the Scott topology. 
Now we obtain our claim 
\[
  (j^! \mathcal{F})(U) 
  = \varprojlim_{x \in U} (j^! \mathcal{F})_x 
  = \varprojlim_{x \in U} \mathcal{F}(\Int U_x) 
  = \mathcal{F}(\Int U). 
\]
\end{proof}

In Proposition \ref{FullyFaithfulAdjointCombo}, 
we establish some basic properties concerning the functors $j^*, j_*$, and $j^!$.
In consequence, we obtain an \emph{adjoint triple} $(j^*, j_*, j^!)$.

\begin{proposition}\label{FullyFaithfulAdjointCombo}
Let $P$ be a continuous poset. 
Then 
\begin{enumerate}[ i)]
  \item the pairs $(j^*, j_*)$ and $(j_*, j^!)$ are adjoint pairs; 
  \item $j_*j^* = \id$ and $j_*j^! = \id$; 
  \item the functors $j^*$ and $j^!$ are fully faithful. 
\end{enumerate}
In particular, $j_*$ is exact. 
\end{proposition}

\begin{proof}
i) The adjointness of $j^*$ and $j_*$ is standard sheaf theory. 
To prove the adjointness of $j_*$ and $j^!$, let $M \in \Fun(\posetC, \Mod)$ and $\mathcal{F} \in \Sh(P^{\sigma})$. 
It is easy to define a map $\Hom(j_*M, \mathcal{F}) \to \Hom(M, j^! \mathcal{F})$ by mapping a morphism $\varphi \colon j_*M \to \mathcal{F}$ to the morphism, which at $p$ is the composition 
\[
  M_p = M(U_p) \overset{\res}{\to} M(\Int U_p) \overset{\varphi}{\to} \mathcal{F}(\Int U_p).
\] 
In order to define the inverse, we need gluing. 
A morphism $f \colon M \to j^! \mathcal{F}$ gives $k$-homomorphisms $f_p \colon M_p \to \mathcal{F}(\Int U_p)$ for all $p \in P$. 
If $s \in M(\Int U_p)$ then the restrictions $s|_{U_x}$ are mapped to $\mathcal{F}(\Int U_x)$ for all $x \in \Int(U_p)$. 
These images are compatible, so we can glue them together, and get a unique section over $\bigcup_{x \gg p} \Int U_x$. 
But this union is $\Int U_p$ by Lemma \ref{Fact:IntU} ii), because $P$ is continuous. 
This defines the inverse, proving i). 

ii) Now 
\[
  (j_* j^! \mathcal{F})(\Int U_p)
  = (j^! \mathcal{F})(\Int U_p) 
  = \mathcal{F}(\Int \Int U_p) 
  = \mathcal{F}(\Int U_p) 
\]
for all $\mathcal{F} \in \Sh(P^{\sigma})$ and $p \in P$. 
Since the collection $\{ \Int(U_p) \}_{p \in P}$ is a basis of the Scott topology, 
we have $j_* j^! \mathcal{F} = \mathcal{F}$, 
and hence, $j_* j^! = \id$. 

On the other hand, 
\[
  \Hom(j_* j^* \mathcal{F}, \mathcal{G}) 
  = \Hom(j^* \mathcal{F}, j^! \mathcal{G}) 
  = \Hom(\mathcal{F}, j_* j^! \mathcal{G}) 
  = \Hom(\mathcal{F}, \mathcal{G}). 
\]
for all $\mathcal{F}, \mathcal{G} \in \Sh(P^{\sigma})$.  
By Yoneda's lemma $j_* j^* \mathcal{F} = \mathcal{F}$, and hence, $j_*j^* = \id$. 

iii) This follows from the two previous items by \cite[IV.3, Theorem 1]{MacLane1971}. 

The functor $j_*$ is right exact, as it is a left adjoint. 
\end{proof}

\begin{definition}\label{overUnderLine}
Let $P$ be a poset, and let $M$ be a persistence module over $P$. 
Set
\[
  \underline{M}_p = \varprojlim_{x \gg p} M_x \quad \text{and} \quad \overline{M}_p = \varinjlim_{x \ll p} M_x 
\]
for all $p \in P$. 
These operations actually define persistence modules:  
the internal morphisms are given by the universal properties of the limit and the colimit. 
Thus, we obtain endofunctors 
\[
  \underline{(-)}, \, \overline{(-)} \colon \Fun(\posetC, \Mod) \to \Fun(\posetC, \Mod). 
\]
Observe that there are canonical morphisms $M \to \underline{M}$ and $\overline{M} \to M$. 
\end{definition}

\begin{remark}\label{Rmk:LineSectionCosection}
Let $M$ be a persistence module over a poset $P$, and let $p \in P$. 
The module $\underline{M}_p$ is now the module of sections $M^{\dagger}(\Int U_p)$. 
In contrast, the module $\overline{M}_p$ coincides with the module of cosections $\hat{M}({\twoheaddownarrow} p)$. 
\end{remark}

\begin{remark}\label{Rmk:UpperLowerBoundary}
Definition \ref{overUnderLine} is related to the upper and lower boundary functors introduced 
by Miller in \cite[Definitions 3.11 and 11.1]{Miller1}. 
Given a polyhedral closed proper convex cone $\gamma \subseteq \Rset^n$ with a non-empty interior and a face $\tau$ of $\gamma$, 
Miller defined the \emph{upper boundary} atop and the \emph{lower boundary} beneath the interior of $\tau$ by the formulas 
\[
  (\delta^{\tau}M)_{\bm{p}} = \varinjlim_{\bm{x} \in \bm{p} - \Int \tau} M_{\bm{x}}
  \quad \text{and} \quad 
  (\partial^{\tau}M)_{\bm{p}} = \varprojlim_{\bm{x} \in \bm{p} + \Int \tau} M_{\bm{x}}
\]
for all $\bm{p} \in \Rset^n$. 
In the case $\tau = \gamma$, Miller's $\delta^{\gamma} M$ and $\partial^{\gamma} M$ are now exactly $\overline{M}$ 
and $\underline{M}$, respectively, 
since the indexing sets are the same.  
\end{remark}

The following Lemma \ref{jConnectionCont}, which connects the adjoint triple $(j^*, j_*, j^!)$ to the functors introduced in Definition \ref{overUnderLine}, will play a crucial role in the sequel.

\begin{lemma}\label{jConnectionCont}
Let $P$ be a continuous poset. 
Then for all persistence modules $M$ over $P$ 
\begin{enumerate}[ i)]
  \item $\underline{M} = j^!j_*M$; 
  \item $\overline{M} = j^*j_* M$. 
\end{enumerate}
In particular, the unit and the counit of the adjunctions $(j_*, j^!)$ and $(j^*, j_*)$ 
correspond to the canonical morphisms $M \to \underline{M}$ and $\overline{M} \to M$, respectively. 
\end{lemma}

\begin{proof}
i) Simply 
\[
  (j^! j_* M)_p = (j_*M)(\Int U_p) = M(\Int U_p) = \varprojlim_{x \gg p} M_x = \underline{M}_p
\]
for all $p \in P$. 

ii) Since the collection $\{ \Int U_p \}_{p \in P}$ forms a basis of the Scott topology, we have 
\[
  (j^* j_*M)_p = (j_* M)_p = \varinjlim_{p \in \Int(U_x)} j_*M(\Int U_x) = \varinjlim_{x \ll p} M(\Int U_x). 
\]
for any $p \in P$. 
On the other hand, 
\[
  \overline{M}_p = \varinjlim_{x \ll p} M_x = \varinjlim_{x \ll p} M(U_x). 
\]
Let us show that these direct limits are isomorphic. 
Firstly, the restrictions $M(U_x) \to M(\Int U_x)$ induce a morphism $\overline{M}_p \to (j^* j_* M)_p$. 
The inverse is constructed similarly. 
For every $x \ll p$, by interpolation property, we have $x \ll y \ll p$ for some $y$. 
Hence, we have the restriction morphism $M(\Int U_x) \to M(U_y)$. 
These induce the desired inverse $(j^* j_* M)_p \to \overline{M}_p$. 
\end{proof}

\subsection{Semi-continuous Modules}

\begin{definition}
Let $P$ be a poset. 
A persistence module $M$ is called \emph{upper semi-continuous}, 
if the canonical morphism $M \to \underline{M}$ is an isomorphism. 
Dually, $M$ is \emph{lower semi-continuous}, if the canonical morphism $\overline{M} \to M$ is an isomorphism. 
We denote the full subcategories of upper semi-continuous modules and lower semi-continuous modules by 
$\Fun^c(\posetC, \Mod)$ and $\Fun_c(\posetC, \Mod)$, respectively. 
\end{definition}

\begin{remark}
Let $P$ be a poset, and let $M$ be a persistence module over $P$. 
If $p \in P$ is an compact element, then 
\[
  \underline{M}_p = \varprojlim_{x \gg p} M_x = M_p 
  \quad \text{and} \quad 
  \overline{M_p} = \varinjlim_{x \ll p} M_x = M_p. 
\]
Now, consider a poset $P$, where every element is compact (for example, a discrete poset), so that $\underline{M} = M = \overline{M}$. 
Therefore, in this case, every persistence module over $P$ is both upper and lower semi-continuous. 
\end{remark}

\begin{proposition}\label{Prop:LowerSemiContCofinal}
Let $P$ be a continuous poset.  
Then a persistence module $M$ over $P$ is lower semi-continuous, 
if for all $p \in P$ there exists $q \ll p$ such that $M$ restricted to the interval $[q, p]$ is a constant module 
(i.e.\ every internal morphism is an isomorphism). 

In particular, the interval modules $k[D]$ and $k[U]$ are lower semi-continuous 
for all Scott-closed $D \subseteq P$ and Scott-open $U \subseteq P$. 
\end{proposition}

\begin{proof}
The direct limit can be computed from a cofinal subset. 
Thus, for all $p \in P$, 
\[
  \overline{M}_p = \varinjlim_{x \ll p} M_x = \varinjlim_{q \ll x \ll p} M_x = M_p, 
\]
which shows the lower semi-continuity of $M$. 
\end{proof}

\begin{proposition}\label{Prop:UpperSemiContCoInitial}
Let $P$ be a continuous poset such that the sets $\Int U_p$ are downward-directed for all $p \in P$.  
Then a persistence module $M$ over $P$ is upper semi-continuous, if for all $p \in P$, there exists $q \gg p$ such that $M$ restricted to the interval $[p, q]$ is constant. 

In particular, the interval modules $k[D]$ and $k[U]$ are upper semi-continuous 
for all open down-sets $D \subseteq \Rset^n$ and closed up-sets $U \subseteq \Rset^n$. 
\end{proposition}

\begin{proof}
This is proven as Proposition \ref{Prop:LowerSemiContCofinal}. 
If $\Int U_p = {\twoheaduparrow} p$ is downward-directed, 
then the coinitial subset $\{ x \in P \mid p \ll x \ll q \}$ is connected 
so that the limit $\varprojlim_{p \ll x \ll q} M_x$ coincides with $M_p$. 
\end{proof}

\begin{remark}
The extra assumption about the sets $\Int U_p$ is necessary in Proposition \ref{Prop:UpperSemiContCoInitial}. 
To see this, let us consider the poset $\Rset^* = \Rset \cup \{ a,b,c \}$, 
where $a < b$, $a < c$ and $x < a$ for all $x \in \Rset$. 
As a constant persistence module $k[\Rset^*]$ trivially satisfies the criterion of Proposition \ref{Prop:UpperSemiContCoInitial}. 
However, it is not upper semi-continuous: Since $\Int U_a = \{ b,c \}$, we have 
\[
    \underline{k[\Rset^*]}_{a} 
    = \varprojlim_{x \in \Int U_{a}} k[\Rset^*]_x 
    = k^2
    \not= k 
    = k[\Rset^*]_{a}. 
\]
This additional assumption is, of course, satisfied in partially ordered abelian groups like $\Rset^n$.
\end{remark}

Conversely, there are easy criteria for non-semi-continuity.

\begin{proposition}\label{Prop:SemiContCounter}
Let $P$ be a continuous poset with $p \not\ll p$ for all $p \in P$, and let $M$ be a persistence module over $P$. 
Then 
\begin{enumerate}
  \item[ i)] $M$ is not lower semi-continuous if $\supp M$ contains a minimal element with respect to the way-below relation; 
  \item[ ii)] $M$ is not upper semi-continuous if $\supp M$ contains a maximal element with respect to the way-below relation. 
\end{enumerate}
In particular, 
\begin{enumerate}
  \item[ iii)] a non-zero finitely generated persistence module is not lower semi-continuous; 
  \item[ iv)] a non-zero finitely co-generated persistence module is not upper semi-continuous. 
\end{enumerate}
\end{proposition}

In this proposition, an element $p \in P$ is called a \emph{minimal} (resp.\ \emph{maximal}) element of a subset $S \subseteq P$ \emph{with respect to the way-below relation}, 
if $p \in S$ and $S \cap {\twoheaddownarrow} p \subseteq \{ p \}$ (resp.\ $S \cap {\twoheaduparrow}p \subseteq \{ p \}$). 
Note that any minimal element is also minimal with respect to the way-below relation, 
since ${\twoheaddownarrow} p \subseteq D_p$ for all $p \in P$.

\begin{proof}
i) Let $p \in \supp M$ be a minimal element with respect to the way-below relation. 
In other words, $\supp M  \cap {\twoheaddownarrow} p \subseteq \{ p \}$. 
Since $p \not\ll p$, we have $\supp M  \cap {\twoheaddownarrow} p = \emptyset$. 
Thus, 
\[
  \overline{M}_p 
  = \varinjlim_{x \ll p} M_x 
  = 0 \not= M_p. 
\]

ii) This is proven as i). 

iii) If $M$ is finitely generated, 
then there exist an exact sequence $\bigoplus_{i = 1}^n k[U_{x_i}] \to M \to 0$ with $x_i \in P$ $(i = 1, \dots, n)$. 
By removing the unnecessary summands, we can assume that the set $\{ x_1, \dots, x_n \}$ is minimal. 
Moreover, it cannot be empty, since $M$ is non-zero. 
Therefore, $\{ x_1, \dots, x_n\}$ is a non-empty finite subset of $\supp M$.
Thus, it contains a minimal element, which is also a minimal element of $\supp M$. 
By i), we conclude that $M$ is not lower semi-continuous. 

iv) This is proven as iii). 
\end{proof}

We state first three simple consequences of Lemma \ref{jConnectionCont}.

\begin{proposition}\label{Prop:SumProdSemiCont}
Let $P$ be a continuous poset. 
Then 
\begin{enumerate}
  \item[ i)] lower semi-continuous modules are closed under all direct sums; 
  \item[ ii)] upper semi-continuous modules are closed under all products. 
\end{enumerate}
\end{proposition}

\begin{proof}
i) Let $M_i$ $(i \in I)$ be lower semi-continuous modules over $P$. 
By Lemma \ref{jConnectionCont}, 
\[
  \overline{\bigoplus_{i \in I} M_i} 
  = j^*j_*\bigoplus_{i \in I} M_i 
  = \bigoplus_{i \in I} j^*j_* M_i 
  = \bigoplus_{i \in I} \overline{M_i} 
  = \bigoplus_{i \in I} M_i, 
\]
since left adjoints commute with colimits, in particular, direct sums. 

ii) This is proven as i). 
\end{proof}

\begin{proposition}\label{Prop:LineComposition}
Let $P$ be a continuous poset. 
Then  
\begin{enumerate}[i) ]
  \item $\underline{\underline{M}} = \underline{M}$; 
  \item $\overline{\overline{M}} = \overline{M}$; 
  \item $\underline{\left(\overline{M}\right)} = \underline{M}$
  \item $\overline{\left( \underline{M} \right)} = \overline{M}$
\end{enumerate}
for all persistence modules $M$ over $P$. 
In particular, the persistence module $\underline{M}$ (resp.\ $\overline{M}$) is upper semi-continuous 
(resp.\ lower semi-continuous) 
for all persistence modules $M$ over $P$. 
\end{proposition}

\begin{proof}
The claim follows from Lemma \ref{jConnectionCont} 
using Proposition \ref{FullyFaithfulAdjointCombo}. 
\end{proof}

If $F \colon \mathcal{A} \to \mathcal{B}$ is a functor, then its \emph{essential image} $\ImF(F)$ is the full subcategory of $\mathcal{B}$ consisting of objects $B \in \mathcal{B}$ such that there is an isomorphism $B \cong F(A)$ for some $A \in \mathcal{A}$.

\begin{proposition}\label{Prop:EssImageJ}
Let $P$ be a continuous poset. 
Then 
\begin{enumerate}
  \item[ i)] $\ImF(j^!) = \Fun^c(\posetC, \Mod)$; 
  \item[ ii)] $\ImF(j^*) = \Fun_c(\posetC, \Mod)$. 
\end{enumerate}
\end{proposition}

\begin{proof}
The proofs are similar, so we only prove i). 
Note first that $\Fun^c(\posetC, \Mod)$ is closed with respect to isomorphic objects. 
If $\mathcal{F} \in \Sh(P^{\sigma})$, then $j^! \mathcal{F}$ is upper semi-continuous, 
since by Lemma \ref{jConnectionCont} i) and Proposition \ref{FullyFaithfulAdjointCombo} ii), $\underline{j^! \mathcal{F}} = j^! j_* j^! \mathcal{F} = j^! \mathcal{F}$. 
On the other hand, if $M$ is an upper semi-continuous module, 
then $M \to \underline{M}$ is an isomorphism. 
Thus, by Lemma \ref{jConnectionCont} i), $M$ is isomorphic to $j^!j_* M$. 
\end{proof}

Our next goal is to show that there exist equivalences between the category of Scott sheaves 
and the categories of semi-continuous modules of both types.
The groundwork has already been laid in Proposition \ref{FullyFaithfulAdjointCombo} and Lemma \ref{jConnectionCont}.

\begin{theorem}\label{EqvUpperSemiCont}
Let $P$ be a continuous poset. 
Then the following categories are equivalent: 
\begin{itemize}
  \item the category of Scott sheaves; 
  \item the full subcategory of upper semi-continuous modules; 
  \item the full subcategory of lower semi-continuous modules. 
\end{itemize}
These equivalences are given by the functors 
\[
\begin{tikzcd}
  \Fun^c(\posetC, \Mod) \arrow["{j_*}", shift left]{r} & 
  \Sh(P^{\sigma}) \arrow["{j^!}", shift left]{l} \arrow["{j^*}", shift right, swap]{r}& 
  \Fun_c(\posetC, \Mod) \arrow["{j_*}", shift right, swap]{l}
\end{tikzcd}
\]
It follows, in particular, that $\Fun^c(\posetC, \Mod)$ and $\Fun_c(\posetC, \Mod)$ are abelian categories. 
\end{theorem}

\begin{proof}
We only show the equivalence between the Scott sheaves and the upper semi-continuous modules. 
The case of lower semi-continuous modules is proved similarly. 
By Proposition \ref{FullyFaithfulAdjointCombo} iii), $j^!$ is fully faithful, 
and every fully faithful functor induces an equivalence from its domain to the essential image 
(see, for example, \cite[IV.4, Theorem 1]{MacLane1971}). 
Thus, the category of Scott sheaves is equivalent to the full subcategory of upper semi-continuous persistence modules by Proposition \ref{Prop:EssImageJ} i). 
\end{proof}

Let $P$ be a poset, and let $k$ be a field. A persistence module $M$ over $P$ is called \emph{finitely encoded} 
if there exists a finite poset $F$, an order-preserving map $f \colon P \to F$, 
and a pointwise finite-dimensional persistence module $N$ over $F$ 
such that $M = f^* N$ (\cite[Definition 2.6]{Miller2}). 
We say that persistence module $M$ is \emph{finitely Scott encoded} if the map $f \colon P \to F$ is also Scott-continuous.

\begin{proposition}\label{FinScottEncodedAreLSemiCont}
Let $P$ be a continuous poset, and let $k$ be a field. 
A finitely Scott encoded persistence module $M$ is lower semi-continuous. 
\end{proposition}

\begin{proof}
There is a Scott-continuous map $f \colon P \to F$ to a finite poset $F$ 
and a pointwise finite-dimensional persistence module $N$ over $F$ such that $M = f^*N$. 
Since the map $f$ is Scott-continuous, the map $f^a$ factors as 
$P^a \overset{j}{\to} P^{\sigma} \overset{f}{\to} F^{\sigma} = F^a$. 
Therefore, also the inverse image factors, and we get $M = (f^a)^*N = j^*f^*N$. 
Now $M$ is lower semi-continuous by Proposition \ref{Prop:EssImageJ} ii). 
\end{proof}

\begin{example}
The so-called finitely determined modules over the poset $\Rset^n$ are good examples of finitely Scott encoded modules. 
Following \cite[Example 2.8]{Miller2}, a persistence module is \emph{finitely determined} 
if it is finitely encoded by some convex projection. 
For $\bm{a}, \bm{b} \in \Zset^n$, we denote by $[\bm{a}, \bm{b}]$ the corresponding closed interval in $\Zset^n$. 
The \emph{convex projection} $\pi \colon \Rset^n \to [\bm{a}, \bm{b}]$ is defined 
by the formula $\pi(\bm{x}) = (\pi_1(x_1), \dots, \pi_n(x_n))$, where 
\[
  \pi_i(x_i) = \max(a_i, \min(\lceil x_i \rceil, b_i))
\]
for all $i = 1, \dots, n$. 
Note that the restriction $\Zset^n \to [\bm{a}, \bm{b}]$ maps a point to the closest point in the box $[\bm{a}, \bm{b}]$. 
It is easy to check that the $\max$- and $\min$-functions are Scott-continuous. 
The ceiling function is Scott-continuous, too. 
Thus, convex projections are Scott continuous as a tuple of Scott-continuous maps. 
It therefore follows from Proposition  \ref{FinScottEncodedAreLSemiCont} that every finitely determined module is lower semi-continuous. 
\end{example}

\subsection{Radical, Socle and Top}

We give our variants of the radical, socle, and top of a module in the following

\begin{definition}\label{Def:RadSocTop}
Let $P$ be a continuous poset. 
The \emph{Scott-radical} of a persistence module $M$ is the submodule $\Rad^{\sigma}(M) = \ImF(\overline{M} \to M)$, 
the \emph{Scott-socle} the submodule $\soc^{\sigma}(M) = \Ker(M \to \underline{M})$, 
and the \emph{Scott-top} the quotient module $\topf^{\sigma}(M) = M / \Rad^{\sigma}(M)$. 
More explicitly, we have 
\[
  \Rad^{\sigma}(M)_p = \bigcup_{x \ll p} \ImF M(x \le p) 
  \quad \text{and} \quad 
  \soc^{\sigma}(M)_p = \bigcap_{x \gg p} \Ker M(p \le x) 
\]
for all $p \in P$. 
Note that the index set in the formula for the Scott-radical is ${\twoheaddownarrow} p$, 
which is directed, since the poset is continuous. 
Therefore, the union of the images is a well-defined submodule of $M_p$. 
\end{definition}

\begin{example}\label{Ex:RadSocTopUpDown}
Let $P$ be a continuous poset. 
It follows immediately from the above formulas and the basic properties of the way-below relation that for all up-sets $U \subseteq P$ and down-sets $D \subseteq P$:
\begin{tasks}[label-width=4ex](2)
  \task[ i)] $\Rad^{\sigma}(k[U]) = k[\Int U]$; 
  \task[ ii)] $\Rad^{\sigma}(k[D]) = k[D]$; 
  \task[ iii)] $\soc^{\sigma}(k[U]) = k[U^{\twoheaduparrow}]$; 
  \task[ iv)] $\soc^{\sigma}(k[D]) = k[D \setminus {\twoheaddownarrow} D]$; 
  \task[ v)] $\topf^{\sigma}(k[U]) = k[U] / k[\Int U] = k[U \setminus \Int U]$; 
  \task[ vi)] $\topf^{\sigma}(k[D]) = 0$; 
\end{tasks}
where in iii), $U^{\twoheaduparrow} = \{ x \in U \mid {\twoheaduparrow}x = \emptyset \}$. 
\end{example}

We now express the Scott-socle and the Scott-top in terms of the Hom and tensor product, respectively:

\begin{proposition}\label{prop:socHom}
Let $P$ be a continuous poset, and let $M$ be a persistence module over $P$. 
Then the following formulas hold for all $p \in P$:
\begin{enumerate}
  \item[ i)] $\soc^{\sigma}(M)_p = \Hom(k[U_p] / k[\Int U_p], M)$; 

  \item[ ii)] $\topf^{\sigma}(M)_p = (k[D_p] /k[{\twoheaddownarrow} p]) \otimes_{\posetC} M$. 
\end{enumerate}
\end{proposition}

\begin{proof}
i) Consider the short exact sequence 
\[
  0 \to k[\Int U_p] \to k[U_p] \to k[U_p] / k[\Int U_p] \to 0. 
\]
An application of the functor $\Hom(-, M)$ yields the exact sequence 
\[
  0 \to \Hom(k[U_p] / k[\Int U_p], M) \to \Hom(k[U_p], M) \to \Hom(k[\Int U_p], M). 
\]
By Lemma \ref{Lemma:SheafCosheaf} i), we have 
\[
  \Hom(k[U_p], M) = M(U_p) = M_p
  \quad \text{and} \quad
  \Hom(k[\Int U_p], M) = M(\Int U_p)
\]
Thus, 
\begin{align*}
  \Hom(k[U_p] / k[\Int U_p], M) 
  &= \Ker(\Hom(k[U_p], M) \to \Hom(k[\Int U_p], M)) \\
  &= \Ker(M(U_p) \to M(\Int U_p)) \\
  &= \soc^{\sigma}(M)_p. 
\end{align*}

ii) This is proven similarly, but by using the formula 
\[
  \hat{M}(D) = \varinjlim_{x \in D} M_x = k[D] \otimes_{\posetC} M, 
\]
of Lemma \ref{Lemma:SheafCosheaf} ii), 
and the exact sequence 
\[
  0 \to k[{\twoheaddownarrow} p]  \to k[D_p] \to k[D_p]/ k[{\twoheaddownarrow} p] \to 0 
\]
over the opposite poset $P^{\op}$. 
\end{proof}

\begin{remark}\label{Ex:GradedRn}
Let $P$ be a continuous poset that is also a partially ordered abelian group. 
Recall that there is an equivalence between persistence modules over $P$ 
and $P$-graded $k[U_0]$-modules. 
Now $k[\Int U_0]$ is a homogeneous ideal of $k[U_0]$. 
Let $M$ be a $P$-graded $k[U_0]$-module. 
It follows directly from the definitions that 
the Scott-radical of $M$ is now the submodule $k[\Int U_0]M$, and the Scott-socle the \emph{annihilator} 
\[
  \soc^{\sigma}(M) = \Ann_M(k[\Int U_0]) = \{ m \in M \mid k[\Int U_0] m = 0 \}. 
\]
On the other hand, the Scott-socle and the Scott-top can be written in terms of the graded Hom and tensor product: 
\begin{enumerate}
  \item[ i)] $\soc^{\sigma}(M) = \grHom_{k[U_0]}(k[U_0] / k[\Int U_0], M)$; 
  \item[ ii)] $\topf^{\sigma}(M) = (k[U_0] /k[ \Int U_0]) \otimes_{k[U_0]} M$. 
\end{enumerate}
Both items follow straight from the previous proposition and the formulas $k[U_p] = k[U_0](-p)$ and $ k[\Int U_p] = k[\Int U_0](-p)$, where $p \in P$. 
For item ii), note also that by reversing the order, we have
\[
   (k[U_0] / k[\Int U_0])^{\op} =  k[D_0]/ k[{\twoheaddownarrow} 0]. 
\]
\end{remark}

We can now prove a variant of Nakayama's lemma and its dual.

\begin{proposition}\label{Prop:Nak}
Let $P$ be a continuous poset with $p \not\ll p$ for all $p \in P$, 
and let $M$ be a persistence module over $P$. 
\begin{enumerate}
  \item[ i)] If $M$ is finitely generated and $\topf^{\sigma}(M) = 0$, then $M = 0$. 

  \item[ ii)] If $M$ is finitely co-generated and $\soc^{\sigma}(M) = 0$, then $M = 0$; 
\end{enumerate}
\end{proposition}

\begin{proof}
We only show ii), the proof for i) being similar. 
Observe first that $\Int U_p \subseteq U_p \setminus \{ p \}$, since $p \not\ll p$. 
The proof is done by contradiction. 
If $M \not= 0$, then, by mirroring the proof of Proposition \ref{Prop:SemiContCounter}, the support $\supp M$ has a maximal element $p_0$. 
Thus, 
\begin{align*}
  \soc^{\sigma}(M)_{p_0} 
  = \bigcap_{x \gg p_0} \Ker M(p_0 \le x) 
  \supseteq \bigcap_{x > p_0} \Ker M(p_0 \le x) 
  = \bigcap_{x > p_0} M_{p_0} 
  = M_{p_0} 
  \not= 0, 
\end{align*}
which is a contradiction. 
\end{proof}

\begin{remark}
The assumption $p \not\ll p$ is essential for the previous theorem. 
For example, over $\Zset^n$ the Scott-top is the zero functor, 
since $\Rad^{\sigma}(M) = M$ for all persistence modules $M$ over $\Zset^n$. 
\end{remark}

\subsection{Homological Properties}

In this subsection, we utilize Proposition \ref{prop:socHom} to gain insight about the Scott-socle and Scott-top, and their derived functors. 
This also leads to results concerning semi-continuous modules.

\begin{proposition}\label{Prop:HereditarySoc}
Let $P$ be a continuous poset, and let $M$ be a persistence module over $P$. 
Then there exist 
\begin{enumerate}
  \item[ i)] an exact sequence $0 \to \soc^{\sigma}(M) \to M \to \underline{M} \to R^1 \soc^{\sigma}(M) \to 0$; 
  \item[ ii)] an isomorphism $R^{n + 1} \soc^{\sigma}(M)_p = H^n(\Int U_p, M^{\dagger})$ for all $n \ge 1$ and $p \in P$. 
\end{enumerate}
Here $R^n \soc^{\sigma}$ and $H^n$ are the right derived functors of $\soc^{\sigma}$ and $\Gamma$, respectively. 
\end{proposition}

\begin{proof}
First, recall that by Proposition \ref{prop:socHom}, we have
\[
  \soc^{\sigma}(M)_p = \Hom(k[U_p] / k[\Int U_p], M) 
\]
for all $p \in P$. 
Also note that by Remark \ref{Rmk:LineSectionCosection} and Lemma \ref{Lemma:SheafCosheaf},
\[
  \underline{M}_p = \Gamma(\Int U_p, M^{\dagger}) = \Hom(k[\Int U_p], M)
\]
In particular, it follows that 
\[
  R^n\soc^{\sigma}(M)_p = \Ext^n(k[U_p] / k[\Int U_p], M) 
  \quad \text{and} \quad 
  H^n(\Int U_p, M^{\dagger}) = \Ext^n(k[\Int U_p], M)
\]
for all $p \in P$ and $n \in \Nset$. 

Now consider the short exact sequence 
\[
  0 \to k[\Int U_p] \to k[U_p] \to k[U_p] / k[\Int U_p] \to 0. 
\]
Here $k[U_p]$ is projective (\cite[Proposition 5]{HoppnerProjectivesFree}), 
implying that $\Ext^n(k[U_p], -)$ vanishes for $n \ge 1$. 
Now the long exact sequence of $\Hom(-,M)$ yields the desired exact sequence and the isomorphism. 
\end{proof}

\begin{proposition}\label{Prop:HereditaryTop}
Let $P$ be a continuous poset, and let $M$ be a persistence module over $P$. 
Then the following hold:
\begin{enumerate}
  \item[ i)] there exist an exact sequence $0 \to L_1 \topf^{\sigma}(M) \to \overline{M} \to M \to \topf^{\sigma}(M) \to 0$; 
  \item[ ii)] $L_{n} \topf^{\sigma}(M) = 0$ for all $n \ge 2$. 
\end{enumerate}
Here $L_n \topf^{\sigma}$ denotes the left derived functor of $\topf^{\sigma}$.  
\end{proposition}

\begin{proof}
As in the proof of Proposition \ref{Prop:HereditarySoc}, we get the exact sequence of i) and an isomorphism 
\[
  L_{n + 1} \topf^{\sigma}(M)_p = H_n({\twoheaddownarrow} p, \hat{M})
\]
for all $n \ge 1$ and $p \in P$. 
By Remark \ref{Rmk:LineSectionCosection}, the cosection module $\hat{M}({\twoheaddownarrow} p)$ is the colimit $\varinjlim_{x \ll p} M_x$. 
The direct limits is an exact functor, which shows that 
\[
  H_n({\twoheaddownarrow} p, \hat{M}) = L_n \varinjlim_{x \ll p} M_x = 0
\]
for all $n \ge 1$. 
\end{proof}

We point out that the higher right derived functors of the Scott-socle often vanish. 
We call a persistence module $M$ over a poset $P$ is \emph{pointwise noetherian} (resp.\ \emph{artinian}) if $M_p$ is noetherian (resp.\ artinian) for all $p \in P$.

\begin{proposition}\label{Prop:SocVanish}
Let $P$ be a continuous poset such that $\Int U_p$ is downward directed for all $p \in P$, and let $M$ be a pointwise artinian persistence module over $P$. 
If $k$ is a noetherian ring or $M$ is pointwise noetherian, 
then 
\[
  R^n \soc^{\sigma}(M) = 0
\]
for all $n \ge 2$. 
\end{proposition}

\begin{proof}
Since $\Int U_p = {\twoheaduparrow} p$ is downward-directed, 
we can use a result of Jensen \cite[Proposition 1.1]{JensenRLim}, which says that the derived functors $R^n \varprojlim_{x \gg p} M_x$ vanish for all $n \ge 1$. 
The isomorphism of Proposition \ref{Prop:HereditarySoc} now gives
\[
  R^{n + 1} \soc^{\sigma}(M)_p 
  = H^{n}(\Int U_p, M^{\dagger}) 
  = R^n \varprojlim_{x \gg p} M_x
  = 0
\]
for all $n \ge 1$ and $p \in P$. 
\end{proof}

\begin{remark}\label{Rmk:LocalCoh}
Let $P$ be a partially ordered abelian group. 
In the exact sequence of Proposition \ref{Prop:HereditarySoc} i), we can replace $\underline{M}$ by the $P$-graded $k[U_0]$-module 
\[
  \bigoplus_{p \in P} \Gamma(\Int U_0, j_*M(p)). 
\]
Indeed, 
\[
  \Gamma(\Int U_0, j_*M(p))
  =\Gamma(\Int U_0, M(p)^{\dagger}) 
  = \varprojlim_{x \gg 0} M(p)_x 
  = \varprojlim_{x \gg p} M_x 
  = \Gamma(\Int U_p, M^{\dagger})
\]
for all $p \in P$. 
Similarly, the isomorphism of ii) can be written as 
\[
  R^{n + 1} \soc^{\sigma}(M) = \bigoplus_{p \in P} H^n(\Int U_0, j_*M(p)). 
\]
We can now see Proposition \ref{Prop:HereditarySoc} as an analogue of the \emph{Serre-Grothendieck correspondence} between sheaf cohomology and local cohomology \cite[Theorem A4.1]{eisenbudCommAlgebra}. 
Proposition \ref{Prop:SocVanish} then also implies the vanishing of sheaf cohomology modules $H^n(\Int U_0, j_*M(p))$ for $n \ge 1$ and $p \in P$. 
\end{remark}

Propositions \ref{Prop:HereditarySoc} and \ref{Prop:HereditaryTop} lead to criteria for semi-continuity.

\begin{proposition}\label{Prop:SemiContHomCrit}
Let $P$ be a continuous poset, and let $M$ be a persistence module over $P$. 
Then 
\begin{enumerate}
    \item[ i)] $M$ is upper semi-continuous if and only if $R^n \soc^{\sigma}(M) = 0$ for $n = 0,1$; 
    \item[ ii)] $M$ is lower semi-continuous if and only if $L_n \topf^{\sigma}(M) = 0$ for $n = 0,1$. 
\end{enumerate}
\end{proposition}

The Scott-socle, the Scott-top, and their first derived functors are related in the following way:

\begin{theorem}\label{Thm:SocTopConnection}
Let $P$ be a continuous poset, and let $M$ be a persistence module over $P$. 
Then 
\begin{enumerate}
  \item[ i)] $R^1 \soc^{\sigma}(\overline{M}) = \topf^{\sigma}(\underline{M})$; 
  \item[ ii)] $L_1 \topf^{\sigma}(\underline{M}) = \soc^{\sigma}(\overline{M})$. 
\end{enumerate}
\end{theorem}

\begin{proof}
By Proposition \ref{Prop:LineComposition}, 
the exact sequence of Proposition \ref{Prop:HereditarySoc} i) for $\overline{M}$ takes the form 
\[
  0 \to \soc^{\sigma}(\overline{M}) \to \overline{M} \to \underline{M} \to R^1 \soc^{\sigma}(\overline{M}) \to 0. 
\]
Similarly, Proposition \ref{Prop:HereditarySoc} ii) gives the exact sequence 
\[
  0 \to L_1 \topf^{\sigma}(\underline{M}) \to \overline{M} \to \underline{M} \to \topf^{\sigma}({\underline{M}}) \to 0. 
\]
It is easy to see that in both exact sequences, the morphism $\overline{M} \to \underline{M}$ 
is the composition of the canonical morphisms $\overline{M} \to M$ and $M \to \underline{M}$. 
\end{proof}

\section{Ephemeral Modules and Scott Sheaves}

\subsection{Ephemeral Modules and Meager Sets}\label{Ephemeral}

\begin{definition}\label{def:Eph}
Let $P$ be a continuous poset. 
A persistence module $M$ over $P$ is called \emph{ephemeral} if $M(p \le q) = 0$ for all $p \ll q$. 
The full subcategory of ephemeral modules over $P$ is denoted by $\Eph(P)$ or just $\Eph$ if the poset is obvious from the context. 
\end{definition}

\begin{remark}
As the way-below relation in $\Rset$ is the normal strict order relation, 
our definition of ephemeral modules extends the definition given in \cite[p.\ 4]{Obs}. 
However, in the case of more general dense linearly ordered poset, our definition is slightly different. 
It follows from the definition of the way-below relation that if a poset $P$ contains a smallest element $\bm{0}$, then $\bm{0} \ll p$ for all $p \in P$. 
In particular, $\bm{0} \ll \bm{0}$. 
Therefore, in this case, the way-below relation and the strict order relations differ. 
This difference will not lead to any relevant deviation. 
\end{remark}

To get a better grasp on ephemeral modules, we introduce the notion of a \emph{meager set}.

\begin{definition}\label{meager}
Let $P$ be a continuous poset. 
A subset $S \subseteq P$ is \emph{meager} if the elements of $S$ are incomparable in the way-below relation. 
That is, if $x \ll y$, then $x \not\in S$ or $y \not\in S$. 
\end{definition}

\begin{proposition}\label{prop:MeagerLemma}
If $P$ is a continuous poset, then 
\begin{enumerate}[ i)]
  \item the set $U \setminus \Int U$ is meager for all up-sets $U \subseteq P$; 
  \item every meager set $S$ is a subset of $U \setminus \Int U$ for some up-set $U \subseteq P$. 
\end{enumerate}
In other words, the maximal meager sets are of the form $U \setminus \Int U$, where $U \subseteq P$ is an up-set. 
\end{proposition}

\begin{proof}
i) Assume that $x \ll y$. 
If $x \in U$, then $y \in \Int U$ so that $y \not\in U \setminus \Int U$. 

ii) We set $U = \bigcup_{x \in S} U_x$ for a given meager set $S$. 
By Lemma \ref{Fact:IntU} iv), $\Int U = \bigcup_{x \in S} \Int U_x$. 
If $p \in S$, then of course $p \in U$. 
We have to show that $p \not\in \Int U$. 
Assuming the contrary, if $p \in \Int U$, then $p \gg x$ for some $x \in S$. 
Since $S$ is meager, we have $p \not\in S$, which is a contradiction. 
\end{proof}

This now leads to a new characterization of ephemerality:

\begin{theorem}\label{meagerSupp}
Let $P$ be a continuous poset. 
A persistence module $M$ over $P$ is ephemeral 
if and only if the support $\supp(m)$ is meager for all $p \in P$ and $m \in M_p = M^{\dagger}(U_p)$. 
Here the support means the standard sheaf theoretic \emph{support of a section} i.e.\  
$\supp(m) = \{ x \in U_p \mid M(p \le x)(m) \not= 0 \}$. 
\end{theorem}

\begin{proof}
Suppose first that $M$ is ephemeral.
If $x \ll y$ and $x \in \supp(m)$, then $M(x \le y)(m) = 0$, and hence, $y \not\in \supp(m)$. 
Therefore, $\supp(m)$ is meager. 

Conversely, suppose that $\supp(m)$ is meager for all $p \in P$ and $m \in M_p$. 
If $p \ll q$, then $p \not\in \supp(m)$ or $q \not\in \supp(m)$. 
Thus, $m = M(p \le p)(m) = 0$ or $M(p \le q)(m) = 0$. 
But in both cases, $M(p \le q)(m) = 0$. 
So $M$ is ephemeral, as desired. 
\end{proof}

The next theorem lays the foundation for the rest of this article.

\begin{theorem}\label{EphWayBelow}
Let $P$ be a continuous poset. 
A persistence module $M$ over $P$ is ephemeral if and only if $j_* M = 0$. 
In other words, $\Eph = \Ker(j_*)$. 
\end{theorem}

\begin{proof}
Let us first assume that $M$ is ephemeral. 
Since the collection $\{ \Int U_p \}_{p \in P}$ is a basis of the Scott topology, 
it is enough to show that $(j_* M)(\Int U_p) = M(\Int U_p) = 0$ for all $p \in P$. 
Let 
\[
  (m_x)_{x \in \Int(U_p)} \in M(\Int U_p) = \varprojlim_{x \gg p} M_x. 
\] 
This means that $M(x \le y)(m_x) = m_y$ for all $x \le y$. 
We have to show that $m_x = 0$ for all $x \in \Int(U_p)$. 
If $x \gg p$, then there exists $y \in P$ such that $p \ll y \ll x$, since a continuous poset has the interpolation property. 
The assumption implies that $M(y \le x) = 0$. 
In particular, $m_x = M(y \le x)(m_{y}) = 0$. 
So $M(\Int U_p) = 0$. 

Let us then assume that $j_* M = 0$.
In other words, $M(\Int U_x) = 0$ for all $x \in P$. 
If $p \ll q$, then $q \in \Int U_p$, because $P$ is continuous. 
Now $U_q \subseteq \Int U_p \subseteq U_p$, since $U_q$ is the smallest up-set, which contains $q$. 
Thus, $M(p \le q)$ factors through $M(\Int U_p) = 0$, 
and hence, $M(p \le q) = 0$. 
\end{proof}

We next characterize ephemeral modules in terms of the associated semi-continuous modules.

\begin{proposition}\label{EphWayBelow2}
Let $P$ be a continuous poset, and let $M$ be a persistence module over $P$. 
Then the following conditions are equivalent: 
\begin{enumerate}[ i)]
  \item $M$ is ephemeral;
  \item $\overline{M} = 0$; 
  \item $\underline{M} = 0$. 
\end{enumerate}
\end{proposition}

\begin{proof}
Now by Proposition \ref{FullyFaithfulAdjointCombo}, Lemma \ref{jConnectionCont} and Theorem \ref{EphWayBelow}: 
\begin{align*}
  M \text{ is ephemeral} 
  \iff j_* M = 0 
  \iff \overline{M} = j^* j_* M = 0. 
\end{align*}
The another equivalence is proven similarly. 
\end{proof}

We use the Scott-radical, socle, and top to describe ephemeral modules.

\begin{proposition}
Let $P$ be a continuous poset, and let $M$ be a persistence module over $P$.  
Then the following conditions are equivalent: 
\begin{enumerate}[ i)]
  \item $M$ is ephemeral; 
  \item $\Rad^{\sigma}(M) = 0$; 
  \item $\soc^{\sigma}(M) = M$; 
  \item $\topf^{\sigma}(M) = M$. 
\end{enumerate}
\end{proposition}

\begin{proof}
The claim immediately follows from the formulas 
\[
  \Rad^{\sigma}(M)_p = \bigcup_{x \ll p} \ImF M(x \le p)
  \quad \text{and} \quad 
  \soc^{\sigma}(M)_p = \bigcap_{x \gg p} \Ker M(p \le x), 
\]
where $p \in P$. 
\end{proof}

\subsection{Indicator Modules}

In this section, we will investigate semi-continuity of indicator modules.
The following Proposition \ref{prop:jReuna} shows that the functor $j_*$ ignores the topological boundary. 
It can be seen in the context of \cite[Proposition 2.20]{Obs} and \cite[Propositions 4.2 and 4.3]{Miller3}.

\begin{proposition}\label{prop:jReuna}
If $P$ is a continuous poset, then 
\begin{enumerate}[i) ]
  \item $j_* k[U] = j_* k[\Int U]$; 
  \item $j_* k[D] = j_* k[\overline{D}]$. 
\end{enumerate}
for all up-sets $U \subseteq P$ and down-sets $D \subseteq P$. 
\end{proposition}

\begin{proof}
i) We have a short exact sequence 
\[
  0 \to k[\Int U] \to k[U] \to k[U] / k[\Int U] \to 0. 
\]
An application of the exact functor $j_*$, gives the short exact sequence 
\[
  0 \to j_* k[\Int U] \to j_* k[U] \to j_*(k[U] / k[\Int U]) \to 0. 
\]
On the right, the quotient $k[U] / k[\Int U]$ is ephemeral, 
since its support $U \setminus \Int U$ is meager by Proposition \ref{prop:MeagerLemma} i). 
Thus, by Theorem \ref{EphWayBelow} $j_*(k[U]/k[\Int U]) = 0$, which proves i). 

ii) Similarly, we have a short exact sequence 
\[
  0 \to k[\overline{D} \setminus D] \to k[\overline{D}] \to k[D] \to 0, 
\]
but now on the left, the set $\overline{D} \setminus D$ is meager so that $j_*(k[\overline{D} \setminus D]) = 0$. 
\end{proof}

As a consequence, we can now prove the following:

\begin{proposition}\label{prop:IntervalIntClosure}
Let $P$ be a continuous poset. 
Then 
\begin{enumerate}[ i)]
  \item $\overline{k[U]} = k[\Int U]$ for every up-set $U \subseteq P$; 
  \item $\overline{k[D]} = k[\overline{D}]$ for every down-set $D \subseteq P$. 
\end{enumerate}
In the case $P = \Rset^n$, we also have 
\begin{enumerate}
  \item[ iii)] $\underline{k[U]} = k[\overline{U}]$ for every up-set $U \subseteq \Rset^n$; 
  \item[ iv)] $\underline{k[D]} = k[\Int D]$ for every down-set $D \subseteq \Rset^n$. 
\end{enumerate}
Here the closure and the interior are taken with respect to the standard topology. 
\end{proposition}

\begin{proof}
i) By Proposition \ref{prop:jReuna}, we have $j_*k[U] = j_*k[\Int U]$. 
Therefore, Lemma \ref{jConnectionCont} and Proposition \ref{Prop:LowerSemiContCofinal} give 
\[
  \overline{k[U]} = j^*j_*k[U] = j^*j_*k[\Int U] = \overline{k[\Int U]} = k[\Int U]. 
\] 

iii) In Remark \ref{Rmrk:RnIntClosBound}, we have asserted that $\Int \overline{U} = \Int U$ 
for all up-sets $U \subseteq \Rset^n$. 
Thus, we get 
\[
  j_*k[U] = j_*k[\Int U] = j_*k[\Int \overline{U}] = j_*k[\overline{U}] 
\]
by utilizing Proposition \ref{prop:jReuna} twice. 
Now, just like in i), we obtain   
\[
  \underline{k[U]} 
  = j^!j_* k[U] 
  = j^! j_* k[\overline{U}] 
  = \underline{k[\overline{U}]} 
  = k[\overline{U}]. 
\]

Items ii) and iv) are  proven similarly as i) and iii). 
\end{proof}

This leads to criteria for semi-continuity of indicator modules.

\begin{theorem}\label{Thm:IntervalIntClosure2}
Let $P$ be a continuous poset, and let $U \subseteq P$ be an up-set and $D \subseteq P$ a down-set. 
Then 
\begin{enumerate}[ i)]
  \item $k[U]$ is lower semi-continuous if and only if $U \subseteq P$ is Scott-open; 
  \item $k[D]$ is lower semi-continuous if and only if $D \subseteq P$ is Scott-closed. 
\end{enumerate}
In the case $P = \Rset^n$, we also have
\begin{enumerate}
  \item[ iii)] $k[U]$ is upper semi-continuous 
    if and only if $U \subseteq \Rset^n$ is closed with respect to the standard topology; 
  \item[ iv)] $k[D]$ is upper semi-continuous 
    if and only if $D \subseteq \Rset^n$ is open with respect to the standard topology. 
\end{enumerate}
\end{theorem}

\begin{proof}
The proofs of all items being similar, we prove only i).
In Proposition \ref{Prop:UpperSemiContCoInitial} we have already showed that $k[U]$ is lower semi-continuous 
for all Scott-open $U \subseteq P$.  
The converse follows from i) of the previous proposition. 
\end{proof}

A persistence module $M$ over a poset $P$ is said to have 
a \emph{down-set co-presentation} if there is an exact sequence $0 \to M \to E \to E'$ 
for some persistence modules $E$ and $E'$ over $P$ that are isomorphic to a direct sum of 
down-set indicator modules $k[D]$ (\cite[Definition 5.11]{Miller2}).

\begin{corollary}\label{Cor:FinCoPresLowerSemiCont}
Let $P$ be a continuous poset, and let $M$ be a persistence module over $P$. 
Assume that $M$ has a down-set co-presentation $0 \to M \to E \to E'$ 
where $E$ and $E'$ are direct sums of Scott-closed indicator modules $k[D]$.
Then $M$ is lower semi-continuous. 

This holds true, in particular, when the sets $D$ are basic down-sets. 
\end{corollary}

\begin{proof}
Since $M$ has a down-set co-presentation, there exists a commutative diagram 
\[
\begin{tikzcd}
  0 \arrow{r} & \overline{M} \arrow{r} \arrow{d} & \overline{E} \arrow{r} \arrow{d} & \overline{E'} \arrow{d} \\
  0 \arrow{r} & M \arrow{r} & E \arrow{r} & E' 
\end{tikzcd}
\]
Here $E$ and $E'$ are direct sums of Scott-closed indicator modules $k[D]$, 
which are lower semi-continuous by Theorem \ref{Thm:IntervalIntClosure2}. 
Now, Proposition \ref{Prop:SumProdSemiCont} shows that $E$ and $E'$ are lower semi-continuous. 
The rows of the diagram are exact, because $\overline{(-)}$ is an exact functor. 
Thus, by the five lemma also $M$ is lower semi-continuous. 
\end{proof}

\begin{remark}
The above argument cannot be used to prove a dual of Corollary \ref{Cor:FinCoPresLowerSemiCont}. 
Firstly, the functor $\underline{(-)}$ is only left exact so that the five lemma does not work. 
Moreover, $k[U_p]$ is not generally an upper semi-continuous module 
(consider, for example, the poset $\Rset \cup \{ \infty \}$). 
\end{remark}

We can relate the Scott-socle and the Scott-top of an indicator module to the topological boundary as follows:

\begin{proposition}\label{Prop:RadSocTopInterval}
Over the poset $\Rset^n$, we have
\begin{enumerate}
  \item[ i)] $\soc^{\sigma}(k[D]) = k[\partial D]$ for all down-sets $D \subseteq \Rset^n$ closed in the standard topology; 
  \item[ ii)] $\topf^{\sigma}(k[U]) = k[\partial U]$ for all up-sets $U \subseteq \Rset^n$ closed in the standard topology; 
  \item[ iii)] $R^1 \soc^{\sigma}(k[U]) = k[\partial U]$ for all up-sets $U \subseteq \Rset^n$ open in the standard topology; 
  \item[ iv)] $L_1 \topf^{\sigma}(k[D]) = k[\partial D]$ for all down-sets $D \subseteq \Rset^n$ open in the standard topology. 
\end{enumerate}
Here $\partial D$ (resp.\ $\partial U$) denotes the boundary of $D$ (resp.\ $U$) in the standard topology. 
\end{proposition}

\begin{proof}
For i) and ii), since $D$ and $U$ are closed with respect to the standard topology, 
Remark \ref{Rmrk:RnIntClosBound} gives the formulas 
$\partial D = D \setminus {\twoheaddownarrow} D$ and $\partial U = U \setminus \Int U$. 
Thus, both claims follow from Example \ref{Ex:RadSocTopUpDown}. 

For iii) and iv), we note that because of Proposition  \ref{prop:IntervalIntClosure} iii) the exact sequence of Proposition  \ref{Prop:HereditarySoc} for $k[U]$ takes the form
\[
  0 \to \soc^{\sigma}(k[U]) \to k[U] \to k[\overline{U}] \to R^1 \soc^{\sigma}(k[U]) \to 0. 
\]
By Example \ref{Ex:RadSocTopUpDown}, $\soc^{\sigma}(k[U]) = 0$. 
Therefore, by utilizing the formulas of Remark \ref{Rmrk:RnIntClosBound}
\[
  R^1 \soc^{\sigma}(k[U]) = k[\overline{U}] / k[U] = k[\overline{U} \setminus U] = k[\partial U]. 
\]
The item iv) is proven as iii), but by using Proposition \ref{Prop:HereditaryTop} instead. 
\end{proof}

\begin{remark}
Over the poset $\Rset^n$, we also have
\begin{enumerate}
  \item[ i)] $R^1 \soc^{\sigma}(k[D]) = 0$ for all down-sets $D \subseteq \Rset^n$; 
 \item[ ii)] $L_1 \topf^{\sigma}(k[U]) = 0$ for all up-sets $U \subseteq \Rset^n$. 
\end{enumerate}
To check, for example, item i), the exact sequence of Proposition \ref{Prop:HereditarySoc} 
shows that $R^1\soc^{\sigma}(k[D])$ is the cokernel of the canonical morphism $k[D] \to \underline{k[D]}$. 
This is an epimorphism by Proposition \ref{prop:IntervalIntClosure} iv), proving the claim. 
\end{remark}

\subsection{The Quotient Category by the Ephemeral Modules}\label{subsec:ScottSheaves}

For a general reference for quotient categories, see, for example, \cite{popescu}. 
Let $\mathcal{A}$ be an abelian category, and let $\mathcal{C}$ be a full subcategory of $\mathcal{A}$. 
One calls $\mathcal{C}$ a \emph{Serre subcategory} 
if for all short exact sequences $0 \to M' \to M \to M'' \to 0$, there is an equivalence: 
\[
  M \in \mathcal{C} \iff M', M'' \in \mathcal{C}. 
\]
If $F \colon \mathcal{A} \to \mathcal{B}$ is an exact functor between abelian categories, 
then the kernel $\Ker(F) = \{ M \in \mathcal{A} \mid F(M) = 0 \}$ is a Serre subcategory of $\mathcal{A}$. 

For a Serre subcategory $\mathcal{C}$ of an abelian category $\mathcal{A}$, the \emph{quotient category} 
$\mathcal{A} / \mathcal{C}$ is defined. 
The quotient category comes with the \emph{quotient functor} $Q \colon \mathcal{A} \to \mathcal{A} / \mathcal{C}$. 
It is exact, essentially surjective, and $\Ker Q = \mathcal{C}$. 
The quotient category $\mathcal{A} / \mathcal{C}$ has the following universal property: 
For any exact functor $F \colon \mathcal{A} \to \mathcal{B}$ such that $\mathcal{C} \subseteq \Ker(F)$, 
there exists a factorization $F = H \circ Q$ for a unique exact functor $H \colon \mathcal{A} / \mathcal{C} \to \mathcal{B}$. 

If $\mathcal{C}$ is a Serre subcategory of an abelian category $\mathcal{A}$, 
then the full subcategories 
\[
  \mathcal{C}^{\perp} = \{ M \in \mathcal{A} \mid \Hom(X,M) = 0 = \Ext^1(X,M) \text{ for all } X \in \mathcal{C} \} 
\]
and 
\[
  {}^{\perp}\mathcal{C} = \{ M \in \mathcal{A} \mid \Hom(M,X) = 0 = \Ext^1(M,X) \text{ for all } X \in \mathcal{C} \}
\]
are called the \emph{right} and \emph{left perpendicular category}, respectively.

\begin{theorem}[{\cite[Theorem 2.1.]{Burban_2017}}]\label{QuotientKernel}
Let $G \colon \mathcal{A} \to \mathcal{B}$ be an exact functor between abelian categories 
with a right adjoint (resp.\ left adjoint) $F \colon \mathcal{B} \to \mathcal{A}$. 
If the natural transformation $G \circ F \to \id_{\mathcal{B}}$ (resp.\ $\id_{\mathcal{B}} \to G \circ F$) 
is an isomorphism and $\mathcal{C} = \Ker G$, then 
\begin{enumerate}[ i)]
  \item $G$ induces an equivalence between $\mathcal{A} / \mathcal{C}$ and $\mathcal{B}$; 
  \item $F$ is full embedding and its essential image $\ImF(F)$ coincides with $\mathcal{C}^{\perp}$ (resp.\ ${}^{\perp}\mathcal{C}$).
    In particular, $\mathcal{B}$ is equivalent to $\mathcal{C}^{\perp}$ (resp. ${}^{\perp}\mathcal{C}$). 
\end{enumerate}
\end{theorem}

A Serre subcategory $\mathcal{C}$ is called \emph{localizing} (resp.\ \emph{colocalizing})
if the \emph{quotient functor} $\mathcal{A} \to \mathcal{A} / \mathcal{C}$ has a right adjoint (resp.\ left adjoint). 
In this case, the quotient functor is called the \emph{localization} (resp.\ \emph{colocalization}) 
and the right adjoint (resp.\ left adjoint) the \emph{section} (resp.\ \emph{cosection}). 
We denote by $\loc$ and $\sect$ (resp.\ $\coloc$ and $\cosec$) the localization and the section 
(resp.\ the colocalization and the cosection) functors. 
The Serre subcategory is \emph{bilocalizing} if it is both localizing and colocalizing. 
Therefore, a bilocalizing subcategory has a natural adjoint triple $(\cosec, \loc, \sec)$ associated to it.

We are ready to prove one of our main results.

\begin{theorem}\label{Eqv1}
If $P$ is a continuous poset, then the following categories are equivalent: 
\begin{enumerate}[ i)]
  \item the quotient category $\Fun(\posetC, \Mod) / \Eph$; 
  \item the category of Scott sheaves $\Sh(P^{\sigma})$; 
  \item the full subcategory of upper semi-continuous modules $\Fun^c(\posetC, \Mod)$; 
  \item the full subcategory of lower semi-continuous modules $\Fun_c(\posetC, \Mod)$. 
\end{enumerate}
Moreover,
\begin{enumerate}
  \item[ a)] $\Eph$ is a bilocalizing subcategory of $\Fun(\posetC, \Mod)$; 
  \item[ b)] $\Fun^c(\posetC, \Mod) = \Eph^{\perp}$; 
  \item[ c)] $\Fun_c(\posetC, \Mod) = {}^{\perp}\Eph$. 
\end{enumerate}
\end{theorem}

\begin{proof}
For the equivalence of i) and ii), observe first that by Proposition \ref{FullyFaithfulAdjointCombo}, the functor $j_*$ is exact and has the right adjoint $j^!$ and the left adjoint $j^*$ that satisfy $j_*j^! = \id$ and $j_*j^* = \id$. 
By considering either one of these adjoints, we now see by Theorem \ref{QuotientKernel} i) that the categories $\Fun(\posetC, \Mod) / \Ker(j_*)$ and $\Sh(P^{\sigma})$ are equivalent. 
But $\Ker(j_*) = \Eph$ by Theorem \ref{EphWayBelow} proving the equivalence of i) and ii). 
Theorem \ref{QuotientKernel} ii) implies the other equivalences, since the essential image of $j^!$ (resp.\ $j^*$) coincides with  $\Fun^c(\posetC, \Mod)$ (resp.\ $\Fun_c(\posetC, \Mod)$) by Proposition \ref{Prop:EssImageJ}. 
Note that the equivalences between ii), iii), and iv) in fact have already been proven in Theorem \ref{EqvUpperSemiCont}. 
The last three claims are now evident. 
\end{proof}

\begin{remark}\label{Rem:recolle}
As a localizing subcategory, 
$\Eph$ has a hereditary torsion theory associated to it (\cite[Lemma 2.1]{Krause1997}).
The corresponding torsion objects are the ephemeral modules. 
By definition, for all persistence modules $M$ over $P$, $\soc^{\sigma}(M)$ is the kernel of the canonical morphism $M \to \underline{M}$. 
It easily follows that $\soc^{\sigma}(M)$ is the largest ephemeral subobject of $M$. 
Therefore, the Scott-socle coincides with the torsion functor. 
Note that the morphism $M \to \underline{M}$ is the unit of the adjuction $(j_*,j^!)$ (see Lemma \ref{jConnectionCont}). 
In turn, $\topf^{\sigma}(M)$ is the largest ephemeral quotient module of $M$. 
In other words, $\Rad^{\sigma}(M)$ is the smallest subobject of $M$ such that the corresponding quotient $\topf^{\sigma}(M) = M / \Rad^{\sigma}(M)$ is ephemeral. 
\end{remark}

\begin{remark}\label{Rem:recolle2}
It is not difficult to see that $(\topf^{\sigma}, \inc, \soc^{\sigma})$ is an adjoint triple.
By Proposition \ref{FullyFaithfulAdjointCombo}, the functors $j^*$ and $j^!$ are fully faithful, and by Theorem \ref{EphWayBelow}, $\ImF \inc = \Eph = \Ker j_*$. 
Thus, we now have a so-called \emph{recollement} 
\[
\begin{tikzcd}
  \Eph \arrow[r,"{\inc}" description] & 
  \Fun(\posetC, \Mod) \arrow[r,"{j_*}" description] \arrow[l,"{\soc}" description, bend left] \arrow[l,"{\topf}" description, bend right] & 
  \Sh(P^{\sigma}) \arrow[l,"{j^*}" description, bend right] \arrow[l,"{j^!}" description, bend left] 
\end{tikzcd}
\]
If we replace the category $\Sh(P^{\sigma})$ by the equivalent category $\Fun^c(\posetC, \Mod)$ (resp.\ $\Fun_c(\posetC, \Mod)$), 
then the adjoint triple $(j^*, j_*, j^!)$ changes to $(\overline{(-)}, \, \underline{(-)}, \inc)$ 
(resp.\ $(\inc, \overline{(-)}, \, \underline{(-)})$) (see Theorem \ref{Eqv1}). 
More on recollements can be found, for example, in \cite{recollements}.
\end{remark}

We can also characterize the indecomposable injective upper semi-continuous modules in the case of $\Rset^n$.

\begin{proposition}\label{Prop:IndecomInjScottSobr2}
Let $k$ be a field. 
In the abelian category of upper semi-continuous modules over $\Rset^n$, 
the indecomposable injective objects are in a one-to-one correspondence with the upward-directed down-sets $D \subseteq \Rset^n$ that are open with respect to the standard topology.  
\end{proposition}

\begin{proof}
Recall first that $\Fun^{c}(\posetC, \Mod) = \Eph^{\perp}$ by Theorem \ref{Eqv1} b). 
We know by \cite[Corollary 2.2.15]{KrauseHomTheoryRep} that the injective objects of $\Eph^{\perp}$ are precisely the injective persistence modules that are in $\Eph^{\perp}$. 
Similarly, the indecomposable objects of $\Eph^{\perp}$ coincide with the indecomposable persistence modules that are in  $\Eph^{\perp}$, since $\Eph^{\perp}$ is closed under direct summands. 
Therefore, we have to consider the persistence modules that are upper semi-continuous, injective, and indecomposable. 
By a result of Höppner \cite[Proposition 1.1]{Hoppner}, 
the indecomposable injective persistence modules are in a one-to-one correspondence with the upward-directed down-sets $D \subseteq \Rset^n$. 
By Theorem \ref{Thm:IntervalIntClosure2} iv), a down-set $k[D]$ is upper semi-continuous if and only if $D$ is open with respect to the standard topology. 
Thus, the claim follows. 
\end{proof}

\section{Interleavings}\label{se:Interleavings}

\subsection{Interleaving with a superlinear family}

A general theory for interleavings has been presented in \cite{Bubenik_2014}. 
Let $P$ be a poset. 
An order-preserving map $T \colon P \to P$ is a \emph{translation} on $P$, if $p \le T(p)$ for all $p \in P$. 
We use the notation $\Trans_P$ for the set of all translations on $P$. 
This set comes with a natural partial order given by 
\[
  T_1 \le T_2 \iff T_1(p) \le T_2(p) \text{ for all } p \in P \qquad (T_1,T_2\in \Trans_P). 
\]
In particular, $\id \le T$ for all $T \in \Trans_P$. 
A \emph{family of translations} is a function $T_{\bullet} \colon \left[0, \infty \right[ \to \Trans_P$. 
It is called \emph{superlinear}, if $T_{\varepsilon} \circ T_{\delta} \le T_{\varepsilon + \delta}$ for all $\varepsilon, \delta \ge 0$.

\begin{example}\label{stdSuperLinFam}
Let $\bm{v} \in \Rset^n$ and $\varepsilon \ge 0$. 
Consider the function 
$T^{\bm{v}}_{\varepsilon} \colon \Rset^n \to \Rset^n, T^{\bm{v}}_{\varepsilon}(\bm{x}) = \bm{x} + \varepsilon \bm{v}$. 
The family $T^{\bm{v}}_{\bullet} = (T^{\bm{v}}_{\varepsilon})_{\varepsilon \ge 0}$ 
is then superlinear for all $\bm{v} \ge \bm{0}$. 
The family corresponding to the point $\bm{v} = (1, \dots, 1)$ is called the \emph{standard superlinear family}.  
Note that $T^{\bm{v}}_{\varepsilon}$ is Scott-continuous for all $\bm{v} \in \Rset^n$ and $\varepsilon \ge 0$. 
\end{example}

Given a persistence module $M \colon \posetC \to \Mod$ and a translation $T \colon P \to P$, 
we write $M T$ for their composition as functors. 
Therefore, $MT$ is also a persistence module $\posetC \to \Mod$ with $(MT)_p = M_{T(p)}$ for all $p \in P$. 
Since $T$ is a translation, there exists a natural morphism $M \to MT$. 

Let $T_{\bullet}$ be a superlinear family of translations on $P$. 
Following \cite[Definition 2.5.1]{Bubenik_2014} 
we say that persistence modules $M$ and $N$ are $\varepsilon$-\emph{interleaved} (\emph{with respect to} $T_{\bullet}$)
if there exist morphisms $f \colon M \to NT_{\varepsilon}$ and $g \colon N \to MT_{\varepsilon}$ 
such that diagram 
\[
\begin{tikzcd}
  M \arrow{r} \arrow{rd} & 
    M T_{\varepsilon} \arrow{r} \arrow{rd} & 
    M T_{\varepsilon} T_{\varepsilon} \\
  N \arrow{r} \arrow{ru} & 
    N T_{\varepsilon} \arrow{r} \arrow{ru} & 
    N T_{\varepsilon} T_{\varepsilon} 
\end{tikzcd}
\]
commutes. 
The \emph{interleaving distance} of $M$ and $N$ is the infimum 
\[
  d^{T}(M, N) 
  = \inf \{ \varepsilon \mid M \text{ and } N \text{ are } \varepsilon \text{-interleaved} \}. 
\]
If $M$ and $N$ are not interleaved for any $\varepsilon$, 
we set $d^{T}(M, N) = \infty$. 
We will omit the superindex if the superlinear family is obvious from the context. 
The interleaving distance is an extended pseudometric 
on the category of persistence modules by \cite[Theorem 2.5.3]{Bubenik_2014}. 

This is the first half of our needs. 
We still need a metric for the category $\Sh(P^{\sigma})$. 
We want to consider the composition $MT$ as the sheaf theoretic inverse image $T^* M$, 
where $M$ is a persistence module and $T$ a translation. 
In order to extend the interleaving distance to Scott sheaves, we must add the assumption that translations are Scott-continuous. 
This idea can be generalized even further by utilizing the so-called spesialization order. 
Let $X$ be a topological space. 
Recall that the \emph{specialization order} is a pre-order on the set $X$ defined by 
\[
  x \le_s y \quad \iff \quad x \in \overline{\{ y \}} \qquad (x,y \in X). 
\]
For more details, see \cite[Subsection 4.2.1]{goubault-larrecq_2013}. 
Now we can define translations for any topological space.

\begin{definition}
Let $X$ be a topological space. 
A continuous map $T \colon X \to X$ is a \emph{translation} if $x \le_s T(x)$ for all $x \in X$. 
We write $\Trans_X$ for the set of all translations on $X$. 
\end{definition}

\begin{proposition}\label{translationOpen}
Let $X$ be a topological space. 
A continuous map $T \colon X \to X$ is a translation if and only if $U \subseteq T^{-1}(U)$ for all open sets $U \subseteq X$. 
In other words, the preimage map $T^{-1} \colon \Open(X) \to \Open(X)$ is a translation, 
where $\Open(X)$ is the poset of open sets ordered by inclusion. 
\end{proposition}

\begin{proof}
Note first that $x \le_s y$ if and only if $y$ is contained in all open neighbourhoods of $x$. 
Suppose now that $T$ is a translation, and let $U \subseteq X$ be an open set. 
If $x \in U$ then $T(x) \in U$, since $x \le_s T(x)$. 
Therefore, $x \in T^{-1}(U)$, which proves the inclusion $U \subseteq T^{-1}(U)$. 
Conversely, suppose that $U \subseteq T^{-1}(U)$ for all open sets $U \subseteq X$ and let $x \in X$. 
So, for all open neighbourhoods $U$ of $x$, we have $x \in T^{-1}(U)$, and hence, $T(x) \in U$. 
This shows that $x \le_s T(x)$. 
\end{proof}

\begin{proposition}
Let $P$ be a poset. 
An order-preserving map $T \colon P \to P$ is a translation
if and only if the corresponding map $T^a \colon P^a \to P^a$ is a translation. 
In other words, $\Trans_P = \Trans_{P^a}$. 
\end{proposition}

\begin{proof}
This is obvious. 
\end{proof}

\begin{proposition}\label{Prop:ScottTranslation}
Let $P$ be a continuous poset. 
A Scott-continuous map $T \colon P \to P$ is a translation
if and only if the corresponding map $T^{\sigma} \colon P^{\sigma} \to P^{\sigma}$ is a translation. 
In particular, $\Trans_{P^{\sigma}} \subseteq \Trans_{P^a}$. 
\end{proposition}

\begin{proof}
It is well known that the closure of a point $x \in P^{\sigma}$ with respect to the Scott topology is the basic down-set $D_x$
(see \cite[Lemma 4.2.7]{goubault-larrecq_2013}). 
Thus, the original partial order and the specialization order are exactly the same relation. 
\end{proof}

Let $X$ be a topological space. 
The set $\Trans_X$ has a natural pre-order given by 
\[
  T_1 \le T_2 \iff T_1(x) \le_s T_2(x) \text{ for all } x \in X \qquad (T_1,T_2\in \Trans_X). 
\]
A \emph{family of translations} is a function $T_{\bullet} \colon \left[0, \infty \right[ \to \Trans_X$. 
It is \emph{superlinear}, if $T_{\varepsilon} \circ T_{\delta} \le T_{\varepsilon + \delta}$ for all $\varepsilon, \delta \ge 0$.

\begin{definition}\label{ScottInterlea}
Let $X$ be a topological space, and let $T_{\bullet} \colon \left[0, \infty \right[ \to \Trans_X$ be a superlinear family. 
Sheaves $\mathcal{F}$ and $\mathcal{G}$ on $X$ are \emph{$\varepsilon$-interleaved} (with respect to $T_{\bullet}$) 
if there exist morphisms $\varphi \colon  \mathcal{F} \to T_{\varepsilon}^* \mathcal{G}$ 
and $\psi \colon  \mathcal{G} \to T_{\varepsilon}^* \mathcal{F}$ such that diagram, 
\begin{equation}\label{eq:Interleaving^*}
\begin{tikzcd}
  \mathcal{F} \arrow{r} \arrow{rd} & 
    T^*_{\varepsilon} \mathcal{F} \arrow{r} \arrow{rd} & 
    T^*_{\varepsilon} T^*_{\varepsilon} \mathcal{F} \\
  \mathcal{G} \arrow{r} \arrow{ru} & 
    T^*_{\varepsilon} \mathcal{G} \arrow{r} \arrow{ru} & 
    T^*_{\varepsilon} T^*_{\varepsilon} \mathcal{G}
\end{tikzcd}
\end{equation}
with natural horizontal maps, commutes. 
The \emph{interleaving distance} of $\mathcal{F}$ and $\mathcal{G}$ is the infimum 
\[
  d^{T}(\mathcal{F}, \mathcal{G}) 
  = \inf \{ \varepsilon \mid \mathcal{F} \text{ and } \mathcal{G} \text{ are } \varepsilon \text{-interleaved} \}. 
\]
If $\mathcal{F}$ and $\mathcal{G}$ are not interleaved for any $\varepsilon$, 
we set $d^{T}(\mathcal{F}, \mathcal{G}) = \infty$. 
\end{definition}

\begin{remark}
The interleaving distance is an extended pseudometric on $\Sh(X)$. 
We omit the proof, since it is done just like for persistence modules (see \cite[Theorem 2.5.3]{Bubenik_2014}). 
Note that $d^T(M, N) = d^T(M^{\dagger}, N^{\dagger})$ for persistence modules $M$ and $N$ on a poset $P$, 
and a superlinear family $T_{\bullet}$ on $P$. 
\end{remark}

\begin{remark}
At first sight, 
Berkouk's and Petit's definition of the interleaving distance, given in \cite[Definition 4.7]{Berkouk_2021} 
seems to differ from our Definition \ref{ScottInterlea}. 
They considered the poset $\Rset^n$ with the partial order defined by a cone $\gamma$, 
and computed all the interleaving distances with respect to the superlinear family $T^{\bm{v}}_{\bullet}$, 
defined in the Example \ref{stdSuperLinFam}. 
However, our definition coincides with their definition. 
To see this, note first that every translation in this superlinear family is a homeomorphism. 
Hence, the inverse image $T^{\bm{v}*}_{\varepsilon}$ 
and the direct image $T^{\bm{v}}_{\varepsilon*}$ are inverses of each other. 
Thus, an application of $T^{\bm{v}}_{\varepsilon*}$ twice to diagram (\ref{eq:Interleaving^*}) 
gives the diagram used in \cite[Definition 4.6]{Berkouk_2021}. 
Conversely, by applying $T^{\bm{v}*}_{\varepsilon}$ twice to that diagram, we return to our diagram. 
\end{remark}

\subsection{Comparing Interleavings on Alexandrov and Scott Topologies}

Let $P$ be a continuous poset. 
For the rest of this article 
we fix a superlinear family of translations $T_{\bullet} \colon \left[0, \infty \right[ \to \Trans_{P^{\sigma}}$. 
Thus, every $T_{\varepsilon}$ must be a Scott-continuous translation. 
By Proposition \ref{Prop:ScottTranslation}, we can think of $T_{\bullet}$ as a family of translations $\left[0, \infty \right[ \to \Trans_{P^{a}} = \Trans_P$. 
For example, in the case $P = \Rset^n$ the family $T_{\bullet}$ could be the standard superlinear family. 
For brevity, 
we use  the notations $d_a$ and $d_{\sigma}$ for the interleaving distances of persistence modules and Scott sheaves, 
respectively. 

For the sequel, we need to put more restrictions on our superlinear family. 
We would like a persistence module $M$ over $P$ to be ephemeral if and only if $d_a(M, 0) = 0$. 
This equivalence does not hold for all superlinear families. 
For example, look at the family $T'_{\bullet} = (\id)_{\varepsilon \ge 0}$. 
For the convenience of our readers, we have collected here all the conditions we will need in the following: 
\begin{enumerate}[\quad TR1:] 
  \item $T_{\varepsilon}$ is an order-isomorphism for all $\varepsilon \ge 0$; 
  \item $T_{\varepsilon}$ is a strong translation for all $\varepsilon > 0$; 
  \item For every $p \ll q$ there exist $\varepsilon > 0$ such that $p \le T_{\varepsilon}(p) \le q$. 
\end{enumerate}
Here an order-preserving map $T \colon P \to P$ is a \emph{strong translation} if $x \ll T(x)$ for all $x \in P$. 

We need condition TR1 for the following lemma.

\begin{lemma}\label{TranslaatioKommutoi}
Let $P$ be a continuous poset. 
If $T \colon P \to P$ is a Scott-continuous order-isomorphism,  
then 
\begin{enumerate}[ i)]
  \item $j^! T^* \mathcal{F} = T^* j^! \mathcal{F}$; 
  \item $j_* T^* M = T^* j_* M$; 
  \item $j^* T^* \mathcal{F} = T^* j^* \mathcal{F}$ 
\end{enumerate}
for all $M \in \Fun(\posetC, \Mod)$ and $\mathcal{F} \in \Sh(P^{\sigma})$. 
\end{lemma}

\begin{proof} 
Note first that $T^{-1}$ is Scott-continuous, since 
\[
  T^{-1}(\sup D) = T^{-1}(\sup T T^{-1}D) = T^{-1}T(\sup T^{-1}D) = \sup (T^{-1} D). 
\]
for all directed sets $D \subseteq P$. 
In particular, $T$ is a homeomorphism, 
and we have 
\[
  (T^{-1})_* = T^* \quad \text{and} \quad (T^{-1})^* = T_*. 
\]
Now 
\begin{enumerate}[i) ]
  \item
    $(j^! T^* \mathcal{F})_p 
    = (T^* \mathcal{F})(\Int U_p) 
    = \mathcal{F}(T(\Int U_p)) 
    = \mathcal{F}(\Int T(U_p)) 
    = (j^! \mathcal{F})(T(U_p))
    = (T^* j^! \mathcal{F})_p$; 

  \item $(j_* T^* M)(U) = (T^* M)(U) = M(T(U)) = (j_* M)(T(U)) = (T^* j_* M)(U)$; 

  \item $j^* T^* \mathcal{F} = (T \circ j)^* \mathcal{F} = (j \circ T)^* \mathcal{F} = T^* j^* \mathcal{F}$ 
\end{enumerate}
for all $p \in P$ and Scott-open $U \subseteq P$. 
\end{proof}

\begin{proposition}\label{colax}
Let $P$ be a continuous poset. 
If the superlinear family of translations $T_{\bullet}$ has TR1, then the functors $j^!, j_*$ and $j^*$ 
preserve $\varepsilon$-interleavings. 
In particular, it follows that 
\begin{align*}
  d_{a}(j^! \mathcal{F}, j^! \mathcal{G}) \le d_{\sigma}(\mathcal{F}, \mathcal{G}), \quad 
  d_{\sigma}(j_* M, j_* N) \le d_{a}(M, N), \quad
  d_{a}(j^*\mathcal{F}, j^* \mathcal{G}) \le d_{\sigma}(\mathcal{F}, \mathcal{G}) 
\end{align*}
for all $M, N \in \Fun(\posetC, \Mod)$ and $\mathcal{F}, \mathcal{G} \in \Sh(P^{\sigma})$. 
\end{proposition}

\begin{proof}
We first prove that $j_*$ preserves $\varepsilon$-interleavings.
If $M$ and $N$ are $\varepsilon$-interleaved, 
then there exist morphisms $f \colon M \to T^*_{\varepsilon} N$ and $g \colon N \to T^*_{\varepsilon} M$, 
which make diagram (\ref{eq:Interleaving^*}) commute. 
By applying the functor $j_*$ to this diagram, 
we get a new diagram and morphisms $j_* f \colon j_* M \to j_* T^*_{\varepsilon} N$, $j_* g \colon j_* N \to j_* T^*_{\varepsilon} M$. 
The functor $j_*$ commutes with translations by the previous lemma. 
So, $j_*M$ and $j_* N$ are $\varepsilon$-interleaved. 
One can similarly show that the functors $j^!$ and $j^*$ preserve $\varepsilon$-interleavings. 

The second claim follows from the first one by the definition of the interleaving distance.
\end{proof}

The following Lemma \ref{InterleaviningLine} and Theorem \ref{stability} explain why the condition TR2 is important.

\begin{lemma}\label{InterleaviningLine}
Let $P$ be a continuous poset, and let $T \colon P \to P$ be a strong translation. 
For every persistence module $M$, there exist natural morphisms 
\begin{enumerate}[ i)]
  \item $\underline{f} \colon M \to T^* \underline{M}$ and $\underline{g} \colon \underline{M} \to T^* M$;  
  \item $\overline{f} \colon M \to T^* \overline{M}$ and $\overline{g} \colon \overline{M} \to T^* M$
\end{enumerate}
that make the diagrams 
\[
\begin{tikzcd}
  M \arrow{r} \arrow{rd}  & T^* M \arrow{r} \arrow{rd} & T^* T^*M \\
  \underline{M} \arrow{r} \arrow{ru} & T^* \underline{M} \arrow{r} \arrow{ru} & T^* T^* \underline{M}
\end{tikzcd}
\quad \text{and} \quad 
\begin{tikzcd}
  M \arrow{r} \arrow{rd}  & T^* M \arrow{r} \arrow{rd} & T^* T^*M \\
  \overline{M} \arrow{r} \arrow{ru} & T^* \overline{M} \arrow{r} \arrow{ru} & T^* T^* \overline{M}
\end{tikzcd}
\]
commute. 
\end{lemma}

\begin{proof}
i) The first morphism is the composition $M \to \underline{M} \to T^* \underline{M}$. 
Since $T$ is a strong translation, we have $T(p) \gg p$ for all $p \in P$. 
Thus, the second morphism at $p$ is the canonical morphism of the limit 
\[
  \underline{M}_p = \varprojlim_{x \gg p} M_x \to M_{T(p)} = (T^*M)_p. 
\]
The diagram commutes simply by the construction of morphisms, which fundamentally are just restriction morphisms. 

ii) The proof is similar to that of i). 
\end{proof}

\begin{theorem}\label{stability}
Let $P$ be a continuous poset. 
If the superlinear family of translations $T_{\bullet}$ has TR2, then 
\begin{enumerate}[ i)]
  \item $d_a(M, \underline{M}) = 0$; 
  \item $d_a(M, \overline{M}) = 0$ 
\end{enumerate}
for all persistence modules $M$ over $P$. 
In particular, $d_a(M, 0) = 0$, when $M$ is ephemeral. 
\end{theorem}

\begin{proof}
To prove i), 
we first observe that $M$ and $\underline{M}$ are $\varepsilon$-interleaved for all $\varepsilon > 0$ by the previous lemma. 
Since the interleaving distance is the infimum of all these epsilons, we get $d_a(M, \underline{M}) = 0$. 
The proof of ii) is similar. 
If $M$ is ephemeral, then $j_* M = 0$ by Theorem \ref{EphWayBelow}. 
Thus, $d_a(M, 0) = d_a(M, j^! j_* M) = 0$ proving the last claim. 
\end{proof}

\subsection{Isometry Theorems}

Let $P$ be a continuous poset. 
We fix again a superlinear family $T_{\bullet} \colon \left[0, \infty \right[ \to \Trans_{P^{\sigma}}$.

\begin{theorem}\label{Isom1}
Let $P$ be a continuous poset. 
If the superlinear family of translations $T_{\bullet}$ has TR1 and TR2, then 
\begin{enumerate}[ i)]
  \item $d_{\sigma}(\mathcal{F}, \mathcal{G}) 
    = d_a(j^* \mathcal{F}, j^* \mathcal{G}) 
    = d_a(j^! \mathcal{F}, j^! \mathcal{G})$; 
  
  \item $d_a(M, N) = d_a(\underline{M}, \underline{N})$; 
  \item $d_a(M, N) = d_a(\overline{M}, \overline{N})$; 
  \item $d_a(M, N) = d_{\sigma}(j_* M, j_*N)$
\end{enumerate}
for all Scott-sheaves $\mathcal{F}, \mathcal{G}$ on $P^{\sigma}$ and persistence modules $M,N$ over $P$. 
\end{theorem}

\begin{proof}
i) We have already shown in Proposition \ref{colax} 
that $d_a(j^! \mathcal{F}, j^! \mathcal{G}) \le d_{\sigma}(\mathcal{F}, \mathcal{G})$. 
To prove the converse inequality, suppose that $j^! \mathcal{F}$ and $j^! \mathcal{G}$ 
are $\varepsilon$-interleaved for some $\varepsilon > 0$. 
Thus, there are morphisms $f \colon j^! \mathcal{F} \to T^*_{\varepsilon} j^! \mathcal{G}$ 
and $g \colon j^! \mathcal{G} \to T^*_{\varepsilon} j^! \mathcal{F}$, 
which make diagram (\ref{eq:Interleaving^*}) commute. 
By utilizing Lemma \ref{TranslaatioKommutoi}, 
we also see that these are morphisms $j^! \mathcal{F} \to  j^! T^*_{\varepsilon} \mathcal{G}$ and 
$j^! \mathcal{G} \to j^! T^*_{\varepsilon}  \mathcal{F}$. 
Since $j^!$ is fully faithful by Proposition \ref{FullyFaithfulAdjointCombo}, 
there exist unique morphisms $\mathcal{F} \to T^*_{\varepsilon}\mathcal{G}$ 
and $\mathcal{G} \to T^*_{\varepsilon} \mathcal{F}$ making diagram (\ref{eq:Interleaving^*}) commute. 
Therefore, also $\mathcal{F}$ and $\mathcal{G}$ are $\varepsilon$-interleaved, 
and we get the equation $d_a(j^! \mathcal{F}, j^! \mathcal{G}) = d_{\sigma}(\mathcal{F}, \mathcal{G})$. 
The equation $d_a(j^* \mathcal{F}, j^* \mathcal{G}) = d_{\sigma}(\mathcal{F}, \mathcal{G})$ is proved exactly in the same way. 

ii) From the triangle inequality and Theorem \ref{stability}, we get 
\[
  d_a(M, N) 
  \le d_a(M, \underline{M}) + d_a(\underline{M}, \underline{N}) + d_a(\underline{N}, N) 
  = d_a(\underline{M}, \underline{N}). 
\]
On the other hand, $j^! j_*$ preserves $\varepsilon$-interleavings by Proposition \ref{colax}, 
hence $d_a(j^! j_* M, j^! j_* N) \le d_a(M, N)$. 

iii) The proof is similar to that of ii). 

iv) From Proposition \ref{colax} we obtain the inequality $d_{\sigma}(j_* M, j_* N) \le d_a(M, N)$. 
On the other hand, by iii), we have $d_a(M, N) = d_a(j^* j_*M, j^* j_* N)$. 
Again, $j^*$ preserves $\varepsilon$-interleavings, so $d_a(j^* j_* M, j^* j_* N) \le d_{\sigma}(j_* M, j_* N)$. 
Therefore, $d_{\sigma}(j_* M, j_* N) = d_a(M, N)$ as desired. 
\end{proof}

\begin{corollary}\label{Cor:Isom}
Let $P$ be a continuous poset.
If the superlinear family of translations $T_{\bullet}$ has TR1, TR2 and TR3, then 
\begin{enumerate}[ i)]
  \item $d_a(M, 0) = 0$ if and only if $M$ is ephemeral; 
  \item $d_{\sigma}(\mathcal{F}, 0) = 0$ if and only if $\mathcal{F} = 0$ 
\end{enumerate}
for all persistence modules $M$ over $P$ and Scott sheaves $\mathcal{F}$ on $P^{\sigma}$. 
\end{corollary}

\begin{proof}
i) By Theorem \ref{stability}, we already know that $d_a(M, 0) = 0$ for all ephemeral modules $M$. 
Conversely, assume that $d_a(M, 0) = 0$. 
Thus, $M$ and $0$ are $\varepsilon$-interleaved for every $\varepsilon > 0$. 
Hence, by diagram (\ref{eq:Interleaving^*}), 
the morphism $M \to T^*_{\varepsilon} T^*_{\varepsilon} M$ is zero for every $\varepsilon > 0$ 
i.e.\ $M(p \le T_{\varepsilon}T_{\varepsilon}(p)) = 0$ for all $p \in P$ and $\varepsilon > 0$. 
We need to show that $M(p \le q) = 0$ for all $p \ll q$. 
By TR3, there exists $\varepsilon > 0$ such that $p \le T_{\varepsilon}(p) \le q$. 
Now by condition TR2, we have $p \ll T_{\varepsilon/2} T_{\varepsilon / 2}(p) \le T_{\varepsilon}(p) \le q$. 
So, $M(p \le q)$ factors through $M_{T_{\varepsilon / 2}T_{\varepsilon/2}(p)}$, which proves i). 

ii) Suppose that $d_{\sigma}(\mathcal{F}, 0) = 0$. 
Since $j_* j^! = \id$, we have 
\[
  0 = d_{\sigma}(\mathcal{F}, 0) = d_{\sigma}(j_* j^! \mathcal{F}, 0) = d_a(j^! \mathcal{F}, 0). 
\]
Thus, by i) the module $j^! \mathcal{F}$ is ephemeral.
Therefore, by Theorem \ref{EphWayBelow} $j_* j^! \mathcal{F} = 0$. 
But again, $j_* j^! = \id$, so $\mathcal{F} = 0$. 
\end{proof}

\bibliographystyle{plain}

\end{document}